\documentclass{article}
\usepackage{amsmath,amsfonts,amssymb,amsthm}
\usepackage{amsmath,amsthm,amssymb,amsfonts,graphicx,amsthm,nicefrac,mathtools}
\usepackage{enumerate}
\usepackage{mathtools}
\DeclarePairedDelimiter{\floor}{\lfloor}{\rfloor}
\usepackage{pstricks-add}
\usepackage[round]{natbib}
\usepackage{graphicx}
\usepackage{paralist}
 
\allowdisplaybreaks

\newcommand{\cs}{{\cal S}}

\newcommand{\cx}{{\cal X}}

\newcommand{\ints}{\mathbb{Z}}

\newcommand{\bsa}{\boldsymbol{a}}
\newcommand{\bsb}{\boldsymbol{b}}
\newcommand{\bsc}{\boldsymbol{c}}
\newcommand{\bsu}{\boldsymbol{u}}
\newcommand{\bsx}{\boldsymbol{x}}
\newcommand{\bsy}{\boldsymbol{y}}
\newcommand{\bsz}{\boldsymbol{z}}

\newcommand{\natu}{\mathbb{N}}
\newcommand{\real}{\mathbb{R}}

\newcommand{\tran}{\mathsf{T}} 

\newcommand{\rd}{\mathrm{\, d}}

\newcommand{\var}{{\mathrm{Var}}}

\newcommand{\vol}{{\mathbf{vol}}}

\newcommand{\wt}{\widetilde}

\newcommand{\cc}{{\mathcal C}}

\renewcommand{\emptyset}{\varnothing}
\renewcommand{\ge}{\geqslant}
\renewcommand{\le}{\leqslant}

\newcommand{\dustd}{\mathbf{U}} 

\newcommand{\e}{{\mathbb{E}}} 

\newcommand{\hk}{{\mathrm{HK}}}

  \newcommand{\BIT}{\begin{itemize}}
\newcommand{\EIT}{\end{itemize}}
\newcommand{\BNUM}{\begin{enumerate}}
\newcommand{\ENUM}{\end{enumerate}}

\def\E{\mathbb{E}} 

\theoremstyle{definition}
\newtheorem{definition}{Definition}

\theoremstyle{plain}
\newtheorem{theorem}{Theorem}

\newtheorem{propo}{Proposition}
\newtheorem{lemma}{Lemma}
\theoremstyle{definition}

\newcommand{\bb}{\mathbb{B}}
\newcommand{\bx}{\mathbb{X}}
\newcommand{\bs}{\mathbb{S}}

\newcommand{\diam}{\mathrm{diam}}
\newcommand{\rect}{\mathrm{rect}}

\newcommand{\sumdot}{\text{\tiny$\bullet$}}

\DeclareMathOperator*{\medcup}{\textstyle \bigcup}

\newcommand{\bsk}{\boldsymbol{k}}
\newcommand{\bst}{\boldsymbol{t}}

\newcommand{\bsell}{\boldsymbol{\ell}}

\newcommand{\otos}{{1{:}s}}

\theoremstyle{definition}

\date{March 2015}
\author{Kinjal Basu\\Stanford University \and Art B. Owen\\Stanford University}
\title{Scrambled geometric net integration over general product spaces}
\begin{document}
\maketitle
\begin{abstract}

Quasi-Monte Carlo (QMC) sampling has been developed for integration over 
$[0,1]^s$ where it has superior accuracy to Monte Carlo (MC) for integrands
of bounded variation.  Scrambled net quadrature gives allows replication based
error estimation for QMC with at least the same accuracy and for smooth
enough integrands even better accuracy than plain QMC.
Integration over triangles, spheres, disks and Cartesian products
of such spaces is more
difficult for QMC because the induced integrand on a unit cube may
fail to have the desired regularity.
In this paper, we present a construction of point sets for numerical integration over 
Cartesian products  of $s$ spaces of dimension $d$, with triangles ($d=2$)
being of special interest.
The point sets are transformations of randomized $(t,m,s)$-nets using recursive
geometric partitions. The resulting integral estimates
are unbiased and their variance is $o(1/n)$ for any integrand in $L^2$ of the product space.
Under smoothness assumptions on the integrand, our randomized
QMC algorithm has variance
$O(n^{-1 - 2/d} (\log n)^{s-1})$, for integration over $s$-fold Cartesian products of $d$-dimensional
domains, compared to $O(n^{-1})$ for ordinary Monte Carlo. 

\vspace*{.1cm}
\end{abstract}

\section{Introduction}
Quasi-Monte Carlo (QMC) sampling  is designed for problems of 
integration over the unit cube $[0,1]^s$.
Sampling over more complicated regions, such as the triangle 
or the sphere is more challenging. 
Measure preserving mappings from the unit cube to those spaces 
work very well for plain Monte Carlo. Unfortunately, the composition of the 
integrand with such a mapping may fail to have even the 
mild smoothness properties that QMC exploits. 

In this paper, we consider quasi-Monte Carlo integration
over product spaces of the form $\cx^s$ where $\cx$ is
a bounded set of dimension $d$. 
We are especially interested in cases with $d=2$
such as triangles, spherical triangles, spheres, hemispheres and disks.
Integration over such sets is important in graphical rendering
\citep{carlo2001state}.
For instance, when $\cx$ is a triangle, an integral of the
form $\int_{\cx^2}f(x_1,x_2)\rd x_1\rd x_2$ describes the potential
for light to leave one triangle and reach another. The function $f$
incorporates the shapes and relative positions of these triangles as
well as whatever lies between them.

Recent work by \cite{basu:owen:2014}
develops two QMC methods for  use in the triangle.
One is a lattice like construction that was the first construction
to attain discrepancy $O(\log(n)/n)$ in that space.
The other is a generalization of the van der Corput sequence
that makes a recursive partition of the triangle.

In this paper, we generalize that van der Corput construction from the unit triangle
to some other sets. We also replace the van der Corput
sequence by digital nets in dimension $s$, to obtain QMC points
in $\cx^s$.
The attraction of digital nets is that they can be randomized in order
to estimate our quadrature error through independent replication
of the estimate.  Those randomizations have the further
advantage of reducing the error by about $O(n^{-1/2})$ compared
to unrandomized QMC, when the integrand is smooth enough.
For a survey of randomized QMC (RQMC) in general,
see \cite{lecu:lemi:2002}.
For an outline of QMC for computer graphics, see \cite{kell:2013}.

We study QMC and RQMC estimates of
$$\mu = \frac1{\vol(\cx)^s}\int_{\cx^s}f(\bsx)\rd\bsx.$$
Our estimates are equal weight rules
\begin{align}\label{eq:ourestimate}
\hat\mu = \frac1n\sum_{i=1}^n f(\bsx_i),\quad\text{where}\quad
\bsx_i = \phi(\bsu_i)
\end{align}
for random points $\bsu_i\in[0,1]^s$.
The transformation $\phi$ maps $[0,1)$ into $\cx$ and is
applied componentwise.
We assume throughout that $\vol(\cx)=1$ whenever
we are integrating over $\cx$, which simplifies several expressions.

For any $f\in L^2(\cx^s)$ we find that $\var(\hat \mu) = o(1/n)$,
so it is asymptotically superior to plain Monte Carlo.
We also find that for each finite $n$, scrambled nets have a variance
bounded by a finite multiple of the Monte Carlo variance, uniformly
over all $f\in L^2(\cx^s)$.

Our main result is that under smoothness conditions on $f$ 
and a sphericity constraint on the partitioning of $\cx$ 
we are able to show that the estimate
\eqref{eq:ourestimate} attains
\begin{equation}\label{eq:mainresult}
\var(\hat{\mu}) = O\left(\frac{(\log n)^{s-1}}{n^{1 + 2/d}}\right)
\end{equation}
when $\bsu_i$ are certain scrambled digital nets.
This  variance rate is obtained via a functional ANOVA decomposition of the integrand.
The case $d=1$ in~\eqref{eq:mainresult} corresponds to the rate
for scrambled net RQMC from \cite{smoovar}.
The primary technical challenge in lifting that result from $d=1$ to $d>1$ is
to show that the composition $f\circ\phi$ is a well-behaved integrand.

The statements above are for integration over $\cx^s$ but our proof
of the main result is for integration over  $\prod_{j=1}^s\cx^{(j)}$
where $\cx^{(j)}\subset\real^d$ are potentially different sets of dimension $d$. 
Some of our results allow different dimensions $d_j$ for the spaces $\cx^{(j)}$.

The rest of the paper is organized as follows. In Section~\ref{sec:background}, we
give background material on digital nets and their scrambling.
In Section~\ref{sec:splits}, we present recursive geometric splits of a region
$\cx\subset\real^d$ and geometric van der Corput sequences based on them.
Section~\ref{sec:nets} generalizes those constructions to Cartesian products
of $s\ge1$ such sets.
Section~\ref{sec:mrotos} presents the ANOVA and multiresolution analysis
of the Cartesian product domains we study.
Those domains are not rectangular and we embed them in rectangular
domains and extend the integrands to rectangular domains as
described in Section~\ref{sec:extension}.  We use both a Whitney extension
and a Sobol' extension and give new results for the latter.
The proof of our main result is in Section~\ref{sec:snetvar}.
Section~\ref{sec:disc} compares the results we obtain to plain QMC
and scrambled nets over $sd$ dimensions. 

We conclude this section by citing some related work on QMC over tensor product spaces.
Tractability results have been obtained for integration over the
$s$-fold product of the hypersphere 
$S^d=\{\bsx\in\real^{d+1}\mid\bsx^\tran\bsx=1\}$
by \cite{kuo:sloa:2005}. \cite{qmc:trac:simplices} obtained such results
for the $s$-fold product of the simplex $T^d
=\{\bsx\in[0,1]^{d}\mid\sum_jx_j\le1\}$.
Those results are non-constructive. For $s$-fold tensor products of $S^2$ 
there is a component-by-component construction
by \cite{hesse2007component}.

\section{Background on QMC and RQMC}\label{sec:background}

Both QMC and ordinary Monte Carlo (MC) correspond to
the case with $d=1$ and $\cx=[0,1]$.
Plain Monte Carlo sampling of $[0,1]^s$ takes $\bsx_i\sim\dustd[0,1]^s$.
The law of large numbers gives $\hat \mu\to \mu$ with probability one
when $f\in L^1$.  If also $f\in L^2$ then the root mean square error
is $\sigma/\sqrt{n}$ where $\sigma^2$ is the variance of $f(\bsx)$
for $\bsx\sim\dustd[0,1]^s$.

QMC sampling improves upon MC by taking $\bsx_i$ more
uniformly distributed in $[0,1]^s$ than random points usually are. 
Uniformity is measured via discrepancy.  
The local discrepancy of $\bsx_1,\dots,\bsx_n\in[0,1]^s$
at point $\bsa\in[0,1]^s$ is
$$
\delta(\bsa)  = \delta(\bsa;\bsx_1,\dots,\bsx_n) 
= \frac1n\sum_{i=1}^n1_{\bsx_i\in[0,\bsa)}-\vol( [0,\bsa)).
$$
The star discrepancy of those points is
$$
D_n^*(\bsx_1,\dots,\bsx_n) = D_n^* = \sup_{\bsa\in[0,1]^d}|\delta(\bsa)|.
$$
The Koksma-Hlawka inequality is
$$
|\hat \mu-\mu| \le D_n^* V_{\hk}(f),
$$
where $V_{\hk}$ is the $s$-dimensional total variation  of $f$
in the sense of Hardy and Krause.
For a detailed account of $V_\hk$ see \cite{variation}.
Numerous constructions are known for which
$D_n^* = O((\log n)^{s-1}/n)$ \citep{nied92} and so
QMC is asymptotically much more accurate than MC
when $V_{\hk}(f)<\infty$.

\subsection{Digital nets and sequences}
Of special interest here are QMC constructions 
known as digital nets \citep{nied87,dick:pill:2010}. 
We describe them through a series of definitions. 
Throughout these definitions $b\ge2$ is an integer
base, $s\ge1$ is an integer dimension and
$\ints_b=\{0,1,\dots,b-1\}$.

\begin{definition}
For $k_j\in\natu$ and $c_j\in\ints_b$ for $j=1,\dots,s$, the set 
$$
\prod_{j=1}^s\Bigl[ \frac{c_j}{b^{k_j}},\frac{c_j+1}{b^{k_j}}\Bigr) 
$$
is a $b$-adic box of dimension $s$. 
\end{definition}

\begin{definition}
For integers $m\ge t\ge0$,
the points $\bsx_1,\dots,\bsx_{b^m}\in[0,1]^s$ 
are a $(t,m,s)$-net in base $b$
if every $b$-adic box of dimension $s$ with volume $b^{t-m}$
contains precisely $b^t$ of the $\bsx_i$. 
\end{definition}

The nets have good equidistribution (low discrepancy) because
boxes $[0,\bsa]$ can be efficiently approximated by unions of
$b$-adic boxes. Digital nets can attain a discrepancy of
$O((\log(n))^{s-1}/n)$. 

\begin{definition}
For integer $t\ge0$, the infinite sequence
$\bsx_1,\bsx_2,\dots\in[0,1]^s$ is a $(t,s)$-sequence in base $b$
if the subsequence $\bsx_{1+r b^m},\dots,\bsx_{(r+1)b^m}$
is a $(t,m,s)$-net in base $b$ for all integers $r\ge0$ and $m\ge t$.
\end{definition}

The $(t,s)$-sequences (called digital sequences) are extensible
versions of $(t,m,s)$-nets.  
They attain a discrepancy of $O((\log(n))^{s}/n)$.
It improves to  $O((\log(n))^{s-1}/n)$
along the subsequence $n=\lambda b^m$ for integers
$m\ge0$ and $1\le\lambda<b$.

\subsection{Scrambling}
Here we consider scrambling of digital nets and give several
theorems for $[0,1)^s$ that we generalize to $\cx^s$.
Let $a\in[0,1)$
have base $b$ expansion
$a=\sum_{k=1}^\infty a_{k}b^{-k}$  where $a_{k}\in\ints_b$. If $a$
has two base $b$ expansions, we take the one with a tail of $0$s,
not a tail of $b-1$s.
We apply random permutations to the digits $a_k$ yielding $x_k\in\ints_b$
and deliver $x=\sum_{k=1}^\infty x_{k}b^{-k}$.
There are many different ways to choose the permutations \citep{altscram}.
Here we present the nested uniform scramble from \cite{rtms}.

In a nested uniform scramble, $x_1=\pi_{\sumdot}(a_1)$
where $\pi_\sumdot$ is a uniform random permutation (all $b!$
permutations equally probable).
Then $x_2 = \pi_{\sumdot a_1}(a_2)$,
$x_3 = \pi_{\sumdot a_1a_2}(a_2)$ and
$x_{k+1}= \pi_{\sumdot a_1a_2\dots a_{k}}(a_{k+1})$ where all of these
permutations are independent and uniform. Notice that the permutation
applied to digit $a_{k+1}$ depends on the previous digits.
A nested uniform scramble of $\bsa=(a_1,\dots,a_s)\in[0,1)^s$ 
applies independent nested uniform scrambles to all $s$
components of $\bsa$, so that
$x_{j,k+1}=\pi_{j\sumdot a_{j1}a_{j2},\dots,a_{jk}}(a_{j,k+1})$.
A nested uniform scramble of $\bsa_1,\dots,\bsa_n\in[0,1)^s$ applies the same set of
permutations to the digits of all $n$ of those points.
Propositions~\ref{prop_unif} and~\ref{prop_preserve} are
from \cite{rtms}.

\begin{propo} 
\label{prop_unif}
Let $\bsa\in[0,1)^s$ and let $\bsx$ be the result of a nested uniform random 
scramble of $\bsa$. Then $\bsx\sim\dustd[0,1)^s$. 
\end{propo}

\begin{propo}
\label{prop_preserve}
If the sequence $\bsa_1,\dots,\bsa_n$ is a $(t,m,s)$-net in base $b$, 
and $\bsx_i$ are a nested uniform scramble of $\bsa_i$,
then $\bsx_i$ are a $(t,m,s)$-net in base $b$ with probability 1. 
Similarly if $\bsa_i$ is a $(t,s)$-sequence in base $b$, 
then $\bsx_i$ is a $(t,s)$-sequence in base $b$ with probability 1.
\end{propo}

In scrambled net quadrature we estimate 
$\mu = \int_{[0,1)^s} f(\bsx)\rd\bsx$ by 
\begin{align}\label {eq:snetquad}
\hat \mu=\frac1n\sum_{i=1}^nf(\bsx_i),
\end{align}
where $\bsx_i$ are a nested uniform scramble of a digital net $\bsa_i$. 

It follows from Proposition~\ref{prop_unif} that
$\e(\hat \mu)=\mu$ for $f\in L^1[0,1)^s$.
When $f\in L^2[0,1)^s$ we can use independent
random replications of the scrambled nets to estimate 
the variance of $\hat \mu$.
If $V_\hk(f)<\infty$ then we obtain 
$\var(\hat \mu) = O( \log(n)^{2(s-1)}/n^2)=O(n^{-2+\epsilon})$
for any $\epsilon>0$ directly from the Koksma-Hlawka inequality.
Surprisingly, scrambling the net has the potential to
improve accuracy:
\begin{theorem}\label{thm:smooth}
Let $f:[0,1]^s\to\real$ with continuous $\frac{\partial^s}{\partial x_1\cdots \partial x_s} f$.
Suppose that 
$\bsx_i$ are a nested uniform scramble of the first $n=\lambda b^m$
points of a $(t,s)$-sequence in base $b$, for $\lambda\in\{1,2,\dots,b-1\}$.
Then for $\hat\mu$ given by~\eqref{eq:snetquad},
$$
\var(\hat \mu) = O\Bigl( \,\frac{\log(n)^{s-1}}{n^3}\,\Bigr) = O(n^{-3+\epsilon})
$$
as $n\to\infty$ for any $\epsilon>0$.
\end{theorem}
\begin{proof}
\cite{smoovar} has this under a Lipschitz condition.
\cite{owen2008local} removes that condition and corrects
a Lemma from the first paper.
\end{proof}

Smoothness is not necessary for 
scrambled nets to attain a better rate than Monte Carlo.
Bounded variation is not even necessary:
\begin{theorem}\label{thm:littleo}
Let $\bsx_1,\dots,\bsx_n$ be a nested uniform scramble of 
a $(t,m,s)$-net in base $b$. Let $f\in L^2([0,1]^s)$. 
Then for $\hat\mu$ given by~\eqref{eq:snetquad},
$$
\var(\hat \mu) = o\Bigl( \frac1n\Bigr) 
$$
as $n\to\infty$. 
\end{theorem}
\begin{proof}
This follows from \cite{snxs}. The case $t=0$
is in \cite{owensinum}. 
\end{proof}

The factor $\log(n)^{s-1}$ is not necessarily small compared to $n^3$
for reasonable sizes of $n$ and large $s$. 
Informally speaking those powers cannot take effect
for scrambled nets until after they are too small to make
the result much worse than plain Monte Carlo:

\begin{theorem}\label{thm:finitebounds}
Let $\bsx_1,\dots,\bsx_n$ be a nested uniform scramble of 
a $(t,m,s)$-net in base $b$. Let $f\in L^2([0,1]^s)$ with 
$\var( f(\bsx))=\sigma^2$ when $\bsx\sim\dustd[0,1]^s$. 
Then for $\hat\mu$ given by~\eqref{eq:snetquad}, 
$$
\var(\hat \mu) \le b^t\Bigl( \frac{b+1}{b-1}\Bigr)^{s-1} \frac{\sigma^2}n. 
$$
If $t=0$, then 
$\var(\hat \mu) \le e\sigma^2/n \doteq 2.718\sigma^2/n$. 
\end{theorem}
\begin{proof}
The first result is in \cite{snxs}, the second is in \cite{owensinum}. 
\end{proof}

\section{Splits and geometric van der Corput sequences}\label{sec:splits}

The van der Corput sequence is constructed as follows.
We begin with an integer $i\ge0$.
We write it as
$i = \sum_{k=1}^{\infty}a_k(i)b^{k-1}$
for digits $a_k(i)\in\ints_b$.
Define the radical inverse function
$\phi_b(i) = \sum_{k=1}^\infty a_k(i)b^{-k}\in[0,1]$.
The van der Corput sequence in base $b$
is $x_i = \phi_b(i-1)$ for $i\ge1$.
It is a $(0,1)$-sequence in base $b$.
The original sequence of
\cite{vand:1935:I,vand:1935:II}
was for base $b=2$.  Any $n$ consecutive van der Corput
points have a discrepancy $O(\log(n)/n)$ where the implied
constant can depend on $b$.

The lowest order base $b$ digit of $i$ determines which
of $b$ subintervals $[a/b,(a+1)/b)$ will contain $x_i$.
The second digit places $x_i$ into one of $b$ sub-subintervals
of the subinterval that the first digit placed it in, and so on.
\cite{basu:owen:2014} used a base $4$ recursive partitioning
of the triangle to generate a triangular van der Corput sequence.
Discrepancy in the triangle is measured through equidistribution
over trapezoidal subsets \citep{bran:colz:giga:trav:2013}.
Triangular van der Corput points have trapezoidal discrepancy 
of $O(n^{-1/2})$.

\subsection{Splits and recursive splittings}

We begin with a notion of splitting sets.
Splits are like partitions, except that we don't
require empty intersections among their parts.

\begin{definition}
Let $\cx\subset\real^d$ have finite and positive volume. 
A $b$-fold split of $\cx$ is a collection of Borel sets $\cx_a$ for $a\in\ints_b$
with $\cx =\cup_{a=0}^{b-1}\cx_a$, $\vol(\cx_a)=\vol(\cx)/b$ for 
$a\in\ints_b$, and $\vol(\cx_a\cap\cx_{a'})=0$ for 
$0\le a<a'<b$. 
\end{definition}


In all cases of interest to us, any overlap between $\cx_a$
and $\cx_{a'}$ for $a\ne a'$ takes place on the boundaries
of those sets. 
The unit interval $[0,1)$ is customarily partitioned 
into subintervals $[a/b,(a+1)/b)$ in QMC.  Handling $\cx=[0,1]$
requires awkward exceptions where the rightmost interval
is closed and all others are half-open.
For general closed sets $\cx$ it could be burdensome
to keep track of which subsets had which parts of their boundaries.
Using splits allows one for example to divide a closed triangle
into four congruent closed triangles.  Of course a partition
is also a valid split.
Our preferred approach uses a randomization 
under which there is probability zero of any sample point appearing 
on a split boundary.


\begin{figure}
\includegraphics[width=\hsize]{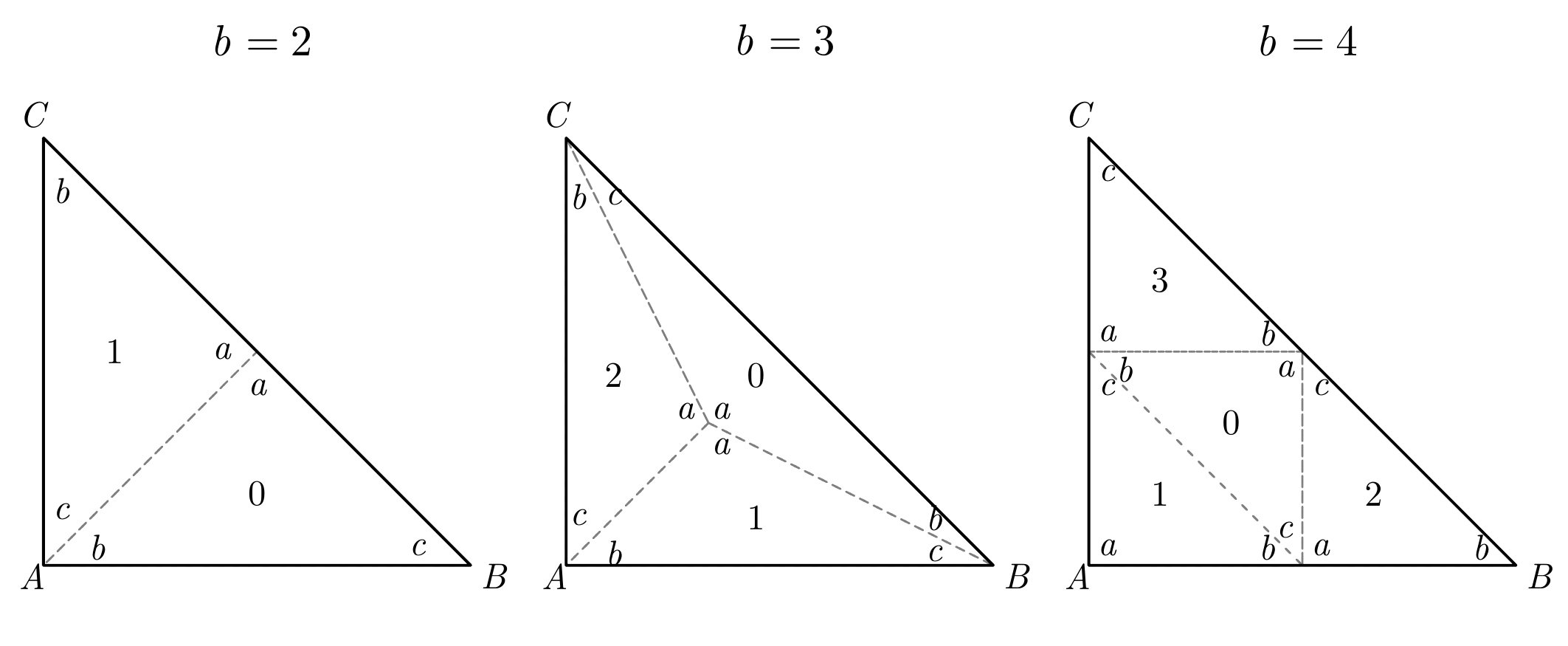}
\caption{\label{fig:base234}
Splits of a triangle $\cx$ for bases $b=2$, $3$ and $4$. 
The subtriangles $\cx_j$ are labeled by the digit 
$j\in \ints_b$.
}
\end{figure}

Figure~\ref{fig:base234} shows a triangle $\cx=\mathrm{ABC}$ split into
subtriangles. The left panel has $b=2$ subtriangles,
the middle panel has $b=3$ and the right panel has $b=4$.
For $b=2$, the vertex labeled `A' is connected to the midpoint
of the opposite side.  The subset ABC$_0$ is the one containing
`B'.  In each case, the new `A' is the mean of the old `B' and `C'.
The new `B' is to the right as one looks from the new A towards
the center of the triangle, and the new `C' is on the left.
An algebraic description is more precise:
Using lower case $abc$ to describe the new $ABC$,
the case described above (base $2$ and digit $0$) has
$$
\begin{pmatrix}
a\\
b\\
c
\end{pmatrix} =
\begin{pmatrix}
0 & 1/2 & 1/2 \\
1 & 0 & 0 \\
0 & 1 & 0
\end{pmatrix}
\begin{pmatrix}
A\\
B\\
C
\end{pmatrix}.
$$
Similar rules apply to the other bases.  From such rules
we may obtain the vertices
of a split at level $k$ by multiplying the original vertices by a sequence
of $k$ $3\times3$ matrices operating on points in the plane.

\begin{definition}
Let $\cx\subset\real^d$ have finite and positive volume. 
A recursive $b$-fold split of $\cx$ is a collection $\bx$ of
sets consisting of $\cx$ and exactly one
$b$-fold split of every set in the collection. The members of $\bx$
are called cells.
\end{definition}

The original set $\cx$ is said to be at level $0$ of the recursive split.
The cells $\cx_{0},\dots,\cx_{b-1}$ of $\cx$ are at level $1$.
A member of a recursive split of $\cx$ is at level $k\ge1$
if it arises after $k$ splits of $\cx$.
The cell $\cx_{a_1,a_2}$ is the subset of $\cx_{a_1}$ corresponding
to $a=a_2$ and similarly, an arbitrary cell at level $k\ge2$
is written $\cx_{a_1,a_2,\dots,a_k}$ for  $a_j\in\ints_b$.

We will need to enumerate all of the cells in a split $\bx$.
For this we write $t=\sum_{j=1}^ka_jb^{j-1}\in\ints_{b^k}$
and then take $\cx_{(k,t)}=\cx_{a_1,a_2,\dots,a_k}$
The cells in the split are now $\cx_{(k,t)}$ for $k\in\natu$
and $t\in\ints_{b^k}$, with $\cx_{(0,0)}=\cx$.

\begin{figure}
\includegraphics[width=\hsize]{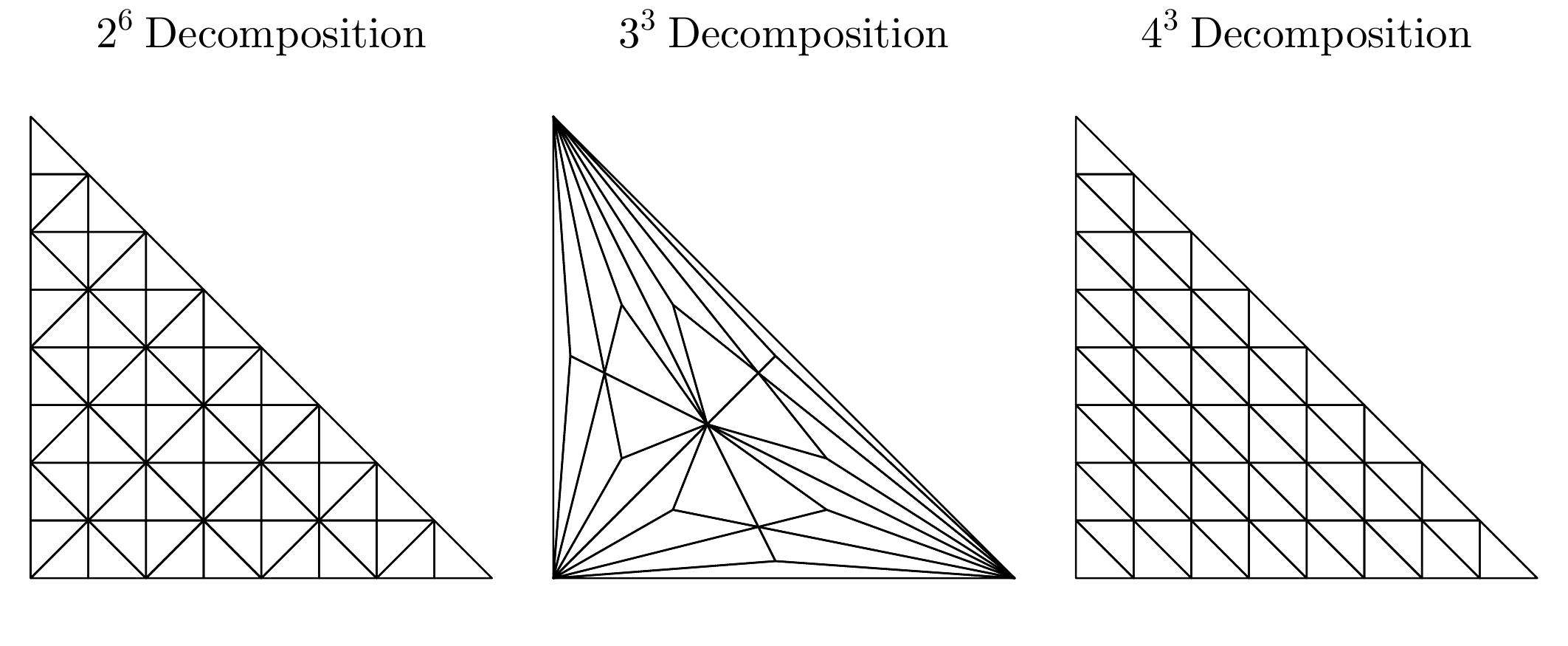}
\caption{\label{fig:base234nesteddecomp} 
The base $b$ splits from Figure~\ref{fig:base234} 
carried out to $k=6$ or $3$ or $4$ levels.  
}
\end{figure}

Figure~\ref{fig:base234nesteddecomp} shows the first
few levels of recursive splits
for each of the splits from Figure~\ref{fig:base234}.
The base $3$ version has elements that become arbitrarily
elongated as $k$ increases.  That is not desirable and our
best results do not apply to such splits.
The base $4$ version was used by \cite{basu:owen:2014}
to define a triangular van der Corput sequence.
The base $2$ version at $6$ levels has a superficially similar
appearance to the base $4$ version at $3$ levels.  But only the
latter has subtriangles similar to~$\cx$.  A linear transformation
to make the parent ABC an equilateral triangle yields congruent 
equilateral subtriangles for $b=4$, while for $b=2$ one gets 
isoceles triangles of two different shapes, some of which are
elongated.

\subsection{Splitting the disk and spherical triangle}
The triangular sets were split in the same way at each level.
Splits can be more general than that, as we illustrate by
splitting a disk.
A convenient way to define a subset
of a disk is in polar coordinates via upper and lower limits
on the radius and an interval of angles (which could
wrap around $2\pi\equiv 0$).
If one alternately splits on angles and radii, the result
is a decomposition of the disk into cells of which
some have very bad aspect ratios, especially those near the center.
\cite{beck:beck:2012} define the aspect ratio of a cell
as the ratio of the length of a circular arc through its
centroid to the length of a radial line through that centroid.
They show decompositions of the disk into $b$ cells
with aspect ratios near one for $b$ as large as several hundred. 
But their decompositions are not recursive.
Figure~\ref{fig:aspectdisks}
shows eight levels in a recursive binary split of the disk into cells.
A cell with aspect ratio larger than one is split along the radial
line through its centroid.  Other cells are split into equal areas
by an arc through the centroid.

\begin{figure}
\includegraphics[width=\hsize]{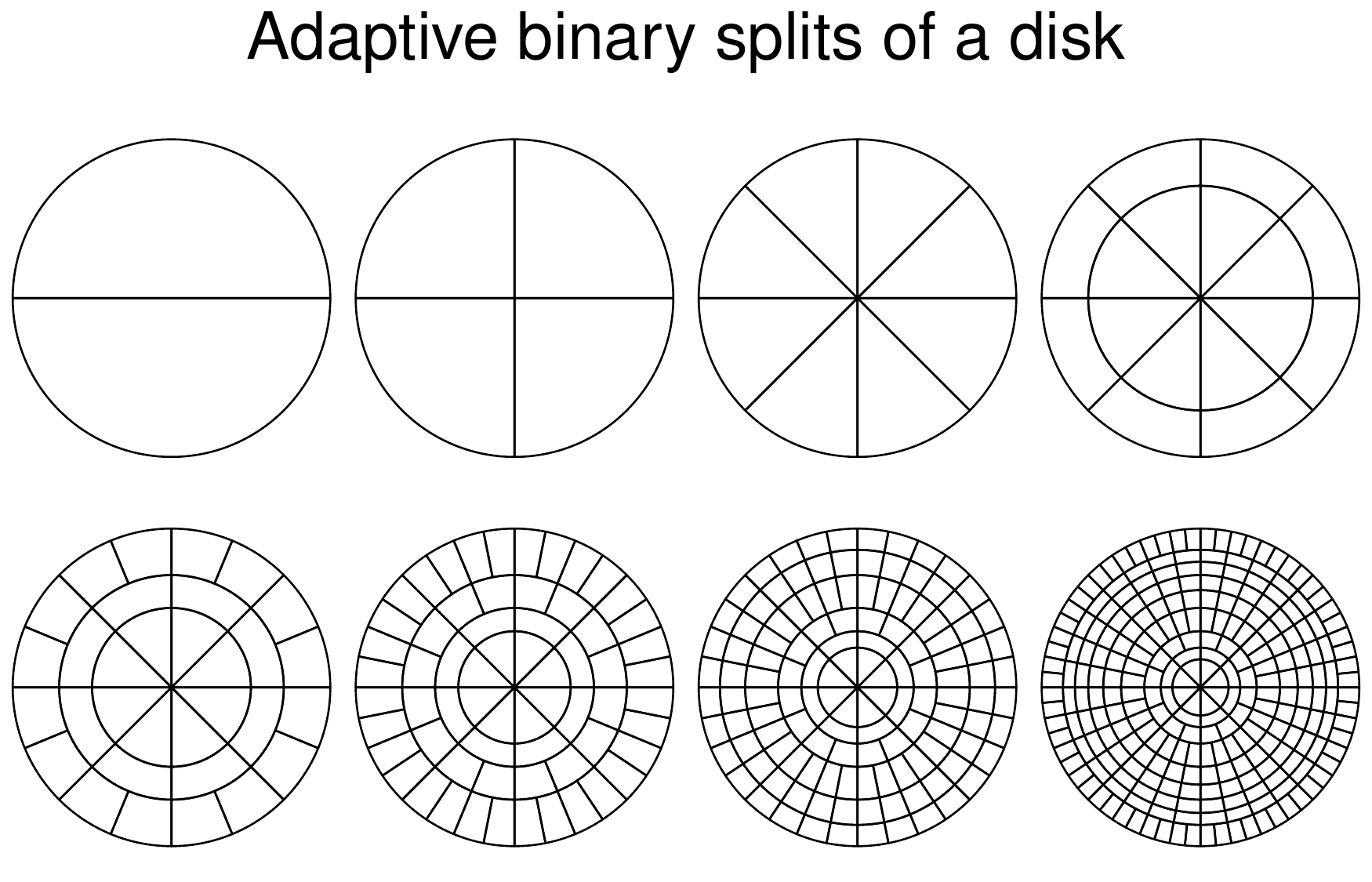}
\caption{\label{fig:aspectdisks}
A recursive binary equal area splitting of the unit disk,
keeping the aspect ratio close to unity.
}
\end{figure}

A spherical triangle is a subset of the sphere in $\real^3$
bounded by $3$ great circles.  By convention, one only
considers spherical triangles for which all internal angles
are less than $\pi$ radians.
The spherical triangle can be split into $b=4$ cells
like the rightmost panel in Figure~\ref{fig:base234}.
If one does so naively, via great circles through midpoints
of the sides of the original triangle, the four cells need not
have equal areas.
\cite{song2002developing} present a four-fold equal area recursive
splitting for spherical triangles, but their inner boundary
arcs are in general small circles.  If one wants a recursive
splitting of great circles into great circles, then it can be
done by generalizing the $b=2$ construction of the 
leftmost panel in Figure~\ref{fig:base234}.
A great circle with vertices ABC can be split into
two by finding a point $P$ on BC so that APB has half
the area of ABC, using the first step of Arvo's algorithm
\cite{arvo:1995}.

\subsection{Geometric van der Corput sequences}

Given a set $\cx$ and a recursive splitting of it in base $b$
we can construct a geometric van der Corput sequence for $\cx$.
The integer $i$ is written in base $b$ as
$i=\sum_{k=1}^\infty a_k(i)b^{k-1}$.
To this $i$ we define a sequence of sets
$$\cx_{i:K} = \cx_{a_1(i),a_2(i),\dots,a_K(i)}.$$
Then $\bsx_i$ is any point in
$\cap_{K=1}^\infty \cx_{i:K}$.
The volume of $\cx_{i:K}$ is $b^{-K}$ which converges
to $0$ as $K\to\infty$.  For most of the constructions
we are interested in, each $\bsx_i$ is a uniquely determined point.  
For decompositions with bad aspect ratios, like the base $3$ decomposition in Figure~\ref{fig:base234nesteddecomp},
some of the set sequences converge to a line segment.
For instance if $i\in\{0,1,2\}$, then $a_k(i)=0$ for $k\ge1$
and that infinite tail of zeros leads to a point $\bsx_i$ on
one of the edges of the triangle.

To get a unique limit $\bsx_i$,
we  use the notion of a sequence of sets converging nicely
to a point. Here is the version from \cite{stro:1994}.
\begin{definition}
The sequence $\cs_k\in\real^d$ of Borel sets for $k\in\natu$ 
converges nicely to $\bsx\in\real^d$ as $k\to\infty$ if
there exists $\alpha<\infty$ and $d$-dimensional cubes $\cc_k$
such that $\bsx\in\cc_k$, $\cs_k\subseteq\cc_k$,
$0<\vol(\cc_k)\le\alpha\vol(\cs_k)$, and
$\lim_{k\to\infty}\diam(\cs_k)=0$.
\end{definition}
A sequence of sets that converges nicely to $\bsx$
cannot also converge nicely to any $\bsx'\ne\bsx$.
We generally assume the following condition.
\begin{definition}\label{def:nice}
A recursive split $\bx$ in base $b$ is convergent if
for every infinite sequence $a_1,a_2,a_3,\dots\in\ints_b$,
the cells $\cx_{a_1,a_2,\dots a_K}$ converges nicely to
a point as $K\to\infty$. That point is denoted
$\lim_{K\to\infty}\cx_{a_1,a_2,\dots,a_K}$.
\end{definition}

In a geometric van der Corput sequence, we take a convergent 
recursive split and choose 
$$\bsx_i=\lim_{M\to\infty}\cx_{a_1(i-1),a_2(i-1),\dots,a_K(i-1),\underbrace{0,0,\dots,0}_M}$$
where $K$ is the last nonzero digit in the expansion of $i-1$ and there are $M\ge1$ zeros above.
For the base $4$ triangular splits, $\bsx_i$ is simply the center point of 
$\cx_{a_1(i-1),a_2(i-1),\dots,a_K(i-1)}$. For $b=2$, $\bsx_i$ is an interior 
but noncentral point.  The recursive split for $b=3$ is not convergent. 

\begin{definition}
Let $\bx$ be a recursive split of $\cx\in\real^d$ in base $b$.
Then $\bx$ satisfies the {\sl sphericity condition} if 
there exists $C<\infty$ such that
$\diam(\cx_{a_1,\dots,a_k})\le Cb^{-k/d}$ holds for all cells
$\cx_{a_1,\dots,a_k}$ in $\bx$.
\end{definition}

A recursive split that satisfies the sphericity condition is necessarily
convergent. The constant $C$ can be as low as $1$ when $d=1$
and the cells are intervals.  The smallest possible $C$ is usually
greater than $1$. We will assume without loss of generality that
$1\le C<\infty$.

\begin{definition}
Given a set $\cx\subset\real^d$ and a convergent recursive split $\bx$
of $\cx$ in base $b$, the $\bx$-transformation of $[0,1)$ is the function
$\phi = \phi_\bx:[0,1)\to\cx$ given by
$\phi(x) = \lim_{K\to\infty} \cx_{x_1,x_2,\dots,x_K}$
where $x$ has the base $b$ representation $0.x_1x_2\dots$.
If $x$ has two representations, then the one with trailing $0$s is used.
\end{definition}

\section{Geometric nets and scrambled geometric nets}\label{sec:nets}

Let $\cx$ be a bounded subset of $\real^d$
with finite nonzero volume.
Here we define digital geometric nets in $\cx^s$ via splittings.
It is convenient at this point to generalize to an $s$-fold
Cartesian product of potentially different spaces, even though they
may  all be copies of the same $\cx$.

For $s\in\natu$, we represent the set $\{1,2,\dots,s\}$ by $1{:}s$. 
For $j\in1{:}s$ we have bounded sets $\cx^{(j)}\subset\real^{d_j}$ with 
$\vol(\cx^{(j)})=1$. 
For sets of indices $u\subseteq1{:}s$, the complement $1{:}s-u$ is denoted by $-u$. 
We use $|u|$ for the cardinality of $u$.
The Cartesian product of $\cx^{(j)}$ for $j\in u$ is denoted $\cx^u$. 
A vector $\bsx\in \cx^{1{:}s}$ has components $\bsx_j\in\cx^{(j)}$. 
The vector in $\cx^u$ with components $\bsx_j$ for $j\in u$ is 
denoted $\bsx_u$. 

A point in $\cx^\otos$ has $\sum_{j=1}^sd_j$ components. 
We write it as 
$\bsx=(\bsx_1,\bsx_2,\dots,\bsx_s)$.  The components in this 
vector of vectors are those of the $\bsx_j$ 
concatenated. 
The notation $\bsx_j$ for $d_j$ consecutive components of $\bsx$
is the same as we use for the $j$'th point in a quadrature rule. The usages 
are different enough that context will make it clear which is intended.

\begin{definition}
For $j=1,\dots,s$, let $\bx_j$ be a recursive split of $\cx^{(j)}$ in a common base $b$.
Denote the cells of $\bx_j$ by $\cx_{j,(k,t)}$ for $k\in\natu$ and $t\in\ints_{b^k}$.
Then a $b$-adic cell for these splits is a Cartesian product of the form
$\prod_{j=1}^s\cx_{j,(k_j,t_j)}$
for integers $k_j\ge0$ and $t_j\in\ints_{b^{k_j}}$.
\end{definition}

\begin{definition}
Let $\cx^{(j)}\subset\real^{d_j}$ have volume $1$ for $j\in\otos$
and let $\bx_j$ be a recursive split of $\cx^{(j)}$ in a common base $b$.
For integers $m\ge t\ge0$,
the points $\bsx_1,\dots,\bsx_{b^m}\in\cx^\otos$ 
are a geometric $(t,m,s)$-net in base $b$
if every $b$-adic cell of volume $b^{t-m}$
contains precisely $b^t$ of the $\bsx_i$. 
They are a weak geometric $(t,m,s)$-net in base $b$
if every $b$-adic cell of volume $b^{t-m}$
contains at least $b^t$ of the $\bsx_i$. 
\end{definition}

Some of the $b$-adic cells can get more than $b^t$
points of a weak geometric $(t,m,s)$-net because the boundaries
of those cells are permitted to overlap. 

\begin{propo}
Let $\bsa_1,\dots,\bsa_n$ be a $(t,m,s)$-net in base $b$.
Let $\bsu_1,\dots,\bsu_n$ be a nested uniform scramble of
$\bsa_1,\dots,\bsa_n$. For $j\in\otos$, let $\bx_j$ be a recursive
base $b$ split of the unit volume set 
$\cx^{(j)}\subset\real^{d_j}$ with transformation $\phi_j$.
Then $\bsz_i = \phi(\bsa_i)$ (componentwise) is a weak geometric
$(t,m,s)$-net in base $b$ and
$\bsx_i = \phi(\bsu_i)$ (componentwise) is a geometric
$(t,m,s)$-net in base $b$ with probability one.
\end{propo}
\begin{proof}
In both cases the transformation applied to half open
intervals places enough points in each $b$-adic cell to make
those points a weak geometric $(t,m,s)$-net.  
The result for scrambled nets follows because each $\bsx_i$
is uniformly distributed and 
$$
\vol\Bigl(\, 
\cx_{j,(k,t)}
\bigcap
\cx_{j,(k,t')}
\Bigr)=0
$$
for all $j\in\otos$, $k\in\ints_{b^k}$ and $0\le t<t'<b^k$.
\end{proof}


\subsection{Measure preservation}

Let $\phi:[0,1) \rightarrow \cx\subset\real^d$ be the function that takes a point in
the unit interval and maps it to $\cx$ according to the convergent recursive split $\bx$. 
We show here that $\phi$ preserves the uniform distribution. 
This is the only section in which we need to distinguish Lebesgue measures
of different dimensions.  To that end,
we use $\lambda_1$ for Lebesgue measure in $\real$ 
and $\lambda_d$ for $\real^d$.

\begin{propo}
\label{prop_trans}
Let $\cx\subset\real^d$ with $\vol(\cx)=1$.
Let $\bx$ be a convergent recursive split of $\cx$ in base $b\ge2$.
Let $\phi$ be the $\bx$-transformation of $[0,1)$ 
and let $A\subseteq\cx$ be a Borel set.  
Then 
\[\lambda_1(\phi^{-1}(A)) = \lambda_d(A)
\]
where $\phi^{-1}(A) = \{ x \in [0,1) \mid \phi(x) \in A \}$.
\end{propo}

\begin{proof}
First, suppose that $A = \cx_{a_1,a_2,\dots,a_k}$ for $a_j\in\ints_b$. 
Then
$\phi^{-1}(A) = [t/b^k, (t+1)/b^k)$ for some $t \in \ints_{b^k}$,
and so
\[ \lambda_d(A) = \frac{1}{b^k} = 
\lambda_1\Bigl( \Bigl[\frac{t}{b^k}, \frac{t+1}{b^k}\Bigr)\Bigr) = 
\lambda_1(\phi^{-1}(A)).
\]

Now let  $A$ be any Borel subset of $\cx$.
Given $\epsilon > 0$, there exists a level $k_1\ge1$ 
and $n$ cells $B_i$ at level $k_1$ of $\bx$ 
such that $A \subseteq \bigcup_{i=1}^{n} B_{i}$ and 
$\lambda_d(\bigcup_{i=1}^{n} B_{i} \setminus A) < \epsilon$.
Similarly, there exists a level $k_2\ge1$  and $m$ cells $C_i$ at level
$k_2$ of $\bx$ such that $\bigcup_{i=1}^{m} C_{i} \subseteq A$ and 
$\lambda_d(A \setminus \bigcup_{i=1}^{m} C_{i}) < \epsilon$.
Thus we get,
\[
\begin{split}
\lambda_d(A) &= \lambda_d\biggl(\,\medcup_{i=1}^{n} B_{i}\biggr) - 
\lambda_d\biggl(\,\medcup_{i=1}^{n} B_{i} \setminus A\biggr) 
\geq  \lambda_d\biggl(\,\medcup_{i=1}^{n} B_{i}\biggr) - \epsilon \\
&= \sum_{i=1}^{n}\lambda_d(B_i) - \epsilon = \sum_{i=1}^{n}\lambda_1(\phi^{-1}(B_i)) - \epsilon = 
\lambda_1\biggl(\phi^{-1}\biggl(\, \medcup_{i=1}^n B_i\biggr)\biggr) - \epsilon \\
&\geq \lambda_1(\phi^{-1}\left(A)\right) - \epsilon,
\end{split}
\]
where the second equality follows because $\phi$ is bijective and
therefore $\phi^{-1}(A) \subseteq \phi^{-1} (\bigcup_{i=1}^n
B_i)$. 
Similarly, $\lambda_d(A)\le \lambda_1(\phi^{-1}(A))+\epsilon$.
Since $\epsilon$ was arbitrary we have the proof.
\end{proof}

Measure preservation extends to the multidimensional
case. The next proposition combines that with uniformity
under scrambling.

\begin{propo}
\label{prop_unif_T}
Let $\cx^{(j)}\subset\real^{d_j}$ with $\vol(\cx^{(j)})=1$ for $j\in\otos$
have convergent recursive splits $\bx_j$ in bases $b_j\ge 2$ with corresponding
transformations $\phi_j$.
Let $\bsa\in[0,1)^s$ and let $x_j$ be a base $b_j$ nested uniform scramble of $a_j$.
Then $\phi(\bsx)=(\phi_1(x_1),\dots,\phi_s(x_s))\sim\dustd(\cx^\otos)$.
\end{propo}

\begin{proof}
By Proposition \ref{prop_unif}, $\bsx$ is uniformly distributed on $[0,1)^s$. 
By Proposition~\ref{prop_trans}, $\phi$
preserves uniform measure.  Thus $\phi_j(x_j)$ are independent
$\dustd(\cx^{(j)})$ random elements.
\end{proof}

\subsection{Results in $L^2$ not requiring smoothness}

Some of the basic properties of scrambled nets
go through for geometric scrambled nets, without
requiring any smoothness of the integrand.
They don't even require that the same base be used
to define both the transformations and the digital net.

\begin{theorem}\label{thm:geomlittleo}
Let $\cx^{(j)}\subset\real^{d_j}$ with $\vol(\cx^{(j)})=1$ for $j=1,\dots,s$.
Let $\bx_j$ be a convergent recursive split of $\cx^{(j)}$ in base $b_j\ge2$
with transformation $\phi_j$.
Let $\bsu_1,\dots,\bsu_n$ be a nested uniform scramble of 
a $(t,m,s)$-net in base $b\ge2$ and let $\bsx_i=\phi(\bsu_i)$
componentwise.
Then for any  $f\in L^2(\cx^\otos)$,
$$
\var(\hat \mu) = o\Bigl( \frac1n\Bigr) 
$$
as $n\to\infty$. 
\end{theorem}
\begin{proof}
Since $f\in L^2(\cx^\otos)$ we have $f\circ\phi\in L^2[0,1]^s$.
Then Theorem~\ref{thm:littleo} applies.
\end{proof}

\begin{theorem}\label{thm:geomfinitebounds}
Under the conditions of Theorem~\ref{thm:geomlittleo},
$$
\var(\hat \mu) \le b^t\Bigl( \frac{b+1}{b-1}\Bigr)^{s-1} \frac{\sigma^2}n,
$$
where $\sigma^2 = \var( f(\bsx))$ for $\bsx\sim\dustd(\cx^\otos)$.
If $t=0$, then 
$\var(\hat \mu) \le e\sigma^2/n \doteq 2.718\sigma^2/n$. 
\end{theorem}
\begin{proof}
Once again, $f\circ\phi\in L^2[0,1]^s$. Therefore
Theorem~\ref{thm:finitebounds} applies.
\end{proof}

\section{ANOVA and multiresolution for $\cx^\otos$}\label{sec:mrotos}

There is a well known analysis of variance (ANOVA) for $[0,1)^s$.  
Here we present the corresponding ANOVA for
$\cx^\otos$. Then we give a multiresolution of $L^2(\cx^\otos)$
adapting the base $b$ wavelet multiresolution in~\cite{owensinum}
for $[0,1)$.

\subsection{ANOVA of $\cx^\otos$}

For $f\in L^2(\cx^{1{:}s})$ the ANOVA decomposition provides
a term for each $u\subseteq1{:}s$. These are defined recursively via
\begin{align}\label{eq:anovau}
f_u(\bsx) = 
\int_{\cx^{-u}} 
\Bigl(f(\bsx) - \sum_{v\subsetneq u}f_v(\bsx)\Bigr)\rd \bsx_{-u}.
\end{align}
The function $f_u$ represents the `effect' of $\bsx_j$ for $j\in u$ above and beyond 
what can be explained by lower order effects of strict subsets $v\subsetneq u$.
While $f_u$ is a function defined on $\cx^{1{:}s}$ its value
only depends on $\bsx_u$. 
By convention $f_\emptyset(\bsx)=\int_{\cx^{1{:}s}} f(\bsx)\rd\bsx = \mu$ for all $\bsx$.
We define variances $\sigma^2_u =\int_{\cx^{1{:}s}}f_u(\bsx)^2\rd\bsx$ for $|u|>0$
and $\sigma^2_\emptyset=0$.
The ANOVA decomposition satisfies $\sum_{|u|>0}\sigma^2_u=\sigma^2$
where $\sigma^2=\int_{\cx^\otos} (f(\bsx)-\mu)^2\rd\bsx$.
It also satisfies $f(\bsx)=\sum_{u\subseteq1{:}s}f_u(\bsx)$
by the definition of $f_\otos$, wherein a $0$-fold integral
of a function leaves it unchanged.
\subsection{Multiresolution}

We begin with a version of base $b$ Haar
wavelets adapted to $\cx\subset\real^d$
using a recursive split $\bx$ of $\cx$ in base $b\ge2$.
Recall that the cells at level $k$ of a split are represented
by one of $\cx_{(k,t)}$ for $0\le t<b^k$. Those cells are
in turn split at level $k+1$ via
$$
\cx_{(k,t)} = \bigcup_{c=0}^{b-1} \cx_{(k,t,c)},
\quad \text{where}\quad \cx_{(k,t,c)} =\cx_{(k+1,bt+c)}.
$$

The multiresolution of $\cx$ in terms of $\bx$
has a function $\varphi(\bsx)=1$ for all $\bsx\in\cx$
as well as functions
\begin{equation}\label{eq:unibasis}
\begin{split}
\psi_{ktc} 
&=  b^{(k+1)/2}1_{\bsx\in \cx_{(k,t,c)}}
-b^{(k-1)/2}1_{\bsx\in \cx_{(k,t)}}\\
& \equiv
b^{(k-1)/2}\bigl( b N_{ktc}(\bsx)-W_{kt}(\bsx)\bigr),
\end{split}
\end{equation}
where $N_{ktc}$ and $W_{kt}$ are indicator functions of
the given narrow and wide cells respectively.
The scaling in~\eqref{eq:unibasis} makes the
norm of $\psi_{ktc}$ independent of $k$:
$\int \psi_{ktc}^2(\bsx)\rd\bsx=(b-1)/b$.

For $f_1,f_2\in L^2(\cx)$ define the inner product
$\langle f_1,f_2\rangle = \int_{\cx}f_1(\bsx) f_2(\bsx)\rd\bsx$.
Then let
$$
f_K(\bsx) 
= \langle f,\varphi\rangle\varphi(\bsx)
+\sum_{k=1}^K\sum_{t=0}^{b^k-1}\sum_{c=0}^{b-1}
\langle f,\psi_{ktc}\rangle \psi_{ktc}(\bsx).
$$

For $\bsx$ belonging to only one cell at level $K+1$,
as almost all $\bsx$ do,
$f_K(\bsx)$ is the average of $f$ over that cell.
By Lebesgue's differentiation theorem, local averages over  
sets that converge nicely to $\bsx$ satisfy  
$$
\lim_{K\to\infty} \frac{\int_{\cs_K} f(\bsx)\rd\bsx}{\vol(\cs_K)}
= f(\bsx),\quad\text{a.e.}
$$ 
for $f\in L^1(\real^d)$.  
So if $\bx$ is convergent, then $\lim_{K\to\infty}f_K(\bsx)=f(\bsx)$
almost everywhere. Thus we may use the representation
\begin{align}\label{eq:tightframe}
f(\bsx) = 
\langle f,\varphi\rangle\varphi(\bsx)
+\sum_{k=1}^\infty\sum_{t=0}^{b^k-1}\sum_{c=0}^{b-1}
\langle f,\psi_{ktc}\rangle \psi_{ktc}(\bsx).
\end{align}
Equation~\eqref{eq:tightframe} resembles a Fourier analysis
with basis functions $\varphi$ and $\psi_{ktc}$.  Unlike the Fourier
case, the functions
$\psi_{ktc}$ and $\psi_{ktc'}$ are not orthogonal. Indeed
$\sum_{c\in\ints_b}\psi_{ktc}=0$ a.e..
Non-orthogonal
bases that nonetheless obey~\eqref{eq:tightframe} are known
as tight frames.

We may extend~\eqref{eq:tightframe} to the multidimensional
setting by taking tensor products.
For $j\in1{:}s$, let $\cx^{(j)}\subset\real^d$ have
recursive split $\bx_j$ in base $b\ge2$. Let the basis
functions be $\varphi_j$ and $\psi_{j(ktc)}$ with narrow
and wide cell indicators $N_{jktc}$ and $W_{jkt}$.
For $u\subseteq1{:}s$, let $\kappa\in\natu^{|u|}$ have
elements $k_j\ge0$ for $j\in u$. 
Similarly let $\tau$ have elements $t_j\in\ints_{b^{k_j}}$ 
and  $\gamma$ have elements $c_j\in\ints_{b}$, both for $j\in u$.
Then for $\bsx\in\cx^{1{:}s}$ define
\begin{align}\label{eq:multiresfunction}
\psi_{u\kappa\tau\gamma}(\bsx)
= \prod_{j\in u}\psi_{jk_jt_jc_j}(\bsx_j)
\prod_{j\not\in u}\varphi_j(\bsx_j).
\end{align}
Our multiresolution of $L^2(\cx^{1{:}s})$ is
\begin{align*}
f(\bsx) &= \sum_{u\subseteq1{:}s}
\sum_{\kappa\mid u}
\sum_{\tau\mid u,\kappa}
\sum_{\gamma\mid u}
\langle\psi_{u\kappa\tau\gamma},f\rangle 
\psi_{u\kappa\tau\gamma}(\bsx)\\ 
 &= \mu + \sum_{|u|>0}
\sum_{\kappa\mid u}
\sum_{\tau\mid u,\kappa}
\sum_{\gamma\mid u}
\langle\psi_{u\kappa\tau\gamma},f\rangle 
\psi_{u\kappa\tau\gamma}(\bsx). 
\end{align*}
The sum over $\kappa$ is over all possible values of $\kappa$
given the subset $u$.  The other sums are similarly over their
entire ranges given the other named variables.

\subsection{Variance and gain coefficients}

Here we study the variance of averages over scrambled
geometric nets. We start with arbitrary points $\bsa_i\in[0,1)^s$.
For now, they need not be from a digital net.
They are given a nested uniform scramble, yielding points
$\bsu_i\in[0,1)^s$. Those points are then mapped to
$\bsx_i\in\cx^{1{:}s}$ using recursive splits in base $b$.

It follows from Proposition \ref{prop_unif_T} that 
\[
\begin{split}
\var(\hat{\mu}) &= \E\biggl( \frac{1}{n^2} \sum_{i=1}^n \sum_{i'=1}^n  
\sum_{|u| > 0}\sum_{\kappa\mid u} \sum_{\tau\mid u,\kappa} \sum_{\gamma\mid u}  
\sum_{|u'| > 0}\sum_{\kappa'\mid u'} \sum_{\tau'\mid u',\kappa'} \sum_{\gamma'\mid u'}  \\
& \quad\qquad\qquad \langle f, \psi_{u\kappa\tau\gamma} \rangle \langle f, \psi_{u'\kappa'\tau'\gamma'} \rangle \psi_{u\kappa\tau\gamma}(\bsx_i)  \psi_{u'\kappa'\tau'\gamma'}(\bsx_{i'})  \biggr).
\end{split}
\]
This formula simplifies due to properties of the randomization. Lemma 4 from \cite{owensinum} shows that if $u \neq u'$ or $\kappa \neq \kappa'$ or $\tau \neq \tau'$, then,
\begin{equation}\label{eq:orthog}
\E( \psi_{u\kappa\tau\gamma}(\bsx_i)  \psi_{u'\kappa'\tau'\gamma'}(\bsx_{i'}) ) = 0.
\end{equation}
Consequently
\begin{equation*}
\var(\hat{\mu}) = \sum_{|u| > 0} \sum_{\kappa\mid u} 
\var\biggl( \frac{1}{n} \sum_{i=1}^n \nu_{u\kappa} (\bsx_i) \biggr),
\end{equation*}
where
\begin{equation*}
\nu_{u\kappa}(\bsx) = \sum_{\tau\mid u,\kappa} \sum_{\gamma\mid u} \langle f, \psi_{u\kappa\tau\gamma} \rangle \psi_{u\kappa\tau\gamma}(\bsx)
\end{equation*}
with $\nu_{\emptyset,()} = \mu$. The function $\nu_{u\kappa}$ is constant within elementary regions of the form
\[\prod_{j \in u} \cx_{j,(k_j,t_j,c_j)}\prod_{j \not\in u} \cx^{(j)}\]
for $0 \leq t_j < b^{k_j }$ and $0\le c_j<b$.

Define
\[ \sigma_{u\kappa}^2 = \int_{\cx^\otos} \nu_{u\kappa}^2(x) \rd \bsx.\]
The multiresolution-based ANOVA decomposition is 
\begin{align}\label{eq:mranova}
\sigma^2 = \int_{\cx^{1{:}s}} (f(\bsx) - \mu)^2 \rd \bsx 
= \sum_{|u| > 0}\sum_{\kappa\mid u} \sigma_{u\kappa}^2
\end{align}
which follows from the orthogonality in~\eqref{eq:orthog}.

The equidistribution properties of $\bsa_1,\dots,\bsa_n$ 
determine the contribution of each $\nu_{u\kappa}$ to $\var(\hat{\mu})$. 
Write $\bsa_i=(a_{i1},\dots,a_{is})$ and define
\[ \Upsilon_{i,i',j,k} = \frac{1}{b-1} \left(b1_{\floor{b^{k+1}a_{ij}} = \floor{b^{k+1}a_{i'j}}}- 1_{\floor{b^{k}a_{ij}} = \floor{b^{k}a_{i'j}}} \right).
\]
For each $|u| > 0$ and $\kappa\in\natu^{|u|}$ define
\[\Gamma_{u,\kappa} = \frac{1}{n}  \sum_{i=1}^n \sum_{i'=1}^n \prod_{j \in u} \Upsilon_{i,i',j,k_j}.
\]
It follows from Theorem 2 of \cite{owensinum} that 
\[
\var(\hat{\mu}) = \frac{1}{n} \sum_{|u| >0} \sum_{\kappa\mid u} \Gamma_{u,\kappa} \sigma_{u\kappa}^2.
\]
We must have $\Gamma_{u,\kappa}\ge0$ because $\var(\hat \mu)\ge0$.
In plain Monte Carlo sampling, $\var(\hat \mu) = {\sigma^2}/n$ which
corresponds to all $\Gamma_{u,\kappa}=1$ (compare~\eqref{eq:mranova}).
The $\Gamma_{u,\kappa}$ are
called `gain coefficients' because they describe variance relative
to plain Monte Carlo.  If the points $\bsa_i$ are carefully
chosen, then many of those coefficients can be reduced and an improvement over
plain Monte Carlo can be obtained. 

If $\bsa_1,\dots,\bsa_n$ are a $(t,m,s)$-net in base $b$, we can put bounds
on the gain coefficients using lemmas from \cite{snxs}. 
In particular,
$\Gamma_{u,\kappa} = 0$ if $m - t\ge |u| + |\kappa|$,
and otherwise
\begin{equation*}
\Gamma_{u,\kappa} 
\leq b^t \left(\frac{b+1}{b-1}\right)^s.
\end{equation*}
Thus finally we have,
\begin{align}
\label{var_multi}
\var(\hat{\mu}) &
\leq \frac{b^t}{n}\left(\frac{b+1}{b-1}\right)^s  \sum_{|u| >0} \sum_{|\kappa| + |u| > m - t }  
\sigma_{u\kappa}^2.
\end{align}

Equation~\eqref{var_multi} shows that the scrambled net variance
depends on the rate at which $\sigma^2_{u\kappa}$ decay as 
$|\kappa|+|u|$ increases. For smooth functions on $[0,1)^s$
they decay rapidly enough to give $\var(\hat \mu) = O( \log(n)^{s-1}/n^3)$.
To get a variance rate on $\cx^{\otos}$ we study the effects of smoothness
on $\sigma^2_{u\kappa}$ for $d$-dimensional spaces $\cx$.

\section{Smoothness and Extension}\label{sec:extension}

Our main results for scrambling geometric nets require
some smoothness of the integrand. We also use some
extensions of the integrand and its ANOVA components to 
rectangular domains.

Let $f$ be a real-valued function on $\cx\subseteq\real^m$. 
The dimension $m$ will usually be $d\times s$, for an $s$-fold
tensor product of a $d$-dimensional region.
For $v\subseteq 1{:}m$, the mixed partial derivative of $f$
taken once with respect to $x_j$ for each $j\in v$ is denoted
$\partial^vf$. By convention $\partial^\emptyset f=f$, as differentiating
a function $0$ times leaves it unchanged.

\subsection{Sobol' extension}
We present the Sobol' extension through a series of definitions.
\begin{definition}
Let $\cx\subseteq\real^m$ for $m\in\natu$.
The function $f:\cx\to\real$ is said to be smooth if $\partial^{1{:}m} f$ is continuous on $\cx$.
\end{definition}

\begin{definition}
Let $\cx\subset\real^m$. The rectangular hull of $\cx$ is
the Cartesian product 
$$
\rect(\cx) = \prod_{j=1}^{m} \bigl[\inf\{x_j\mid \bsx\in\cx\},
\sup\{x_j\mid \bsx\in\cx\}\bigr],
$$
which we also call a bounding box.
For two points $\bsa,\bsb\in\real^m$
we write $\rect[\bsa,\bsb]$ as a shorthand for $\rect[ \{\bsa,\bsb\}]$.
\end{definition}

For later use, we note that
\begin{align}\label{eq:diamrect}
\diam( \rect(\cx))\le\sqrt{d}\times\diam(\cx),
\end{align}
for $\cx\subset\real^d$.

\begin{definition}
A closed set $\cx \subseteq \mathbb{R}^{m}$ with non-empty interior is said to be \textit{Sobol' extensible} if there exists a point $\bsc \in \cx$ such that $\bsz \in \cx$ implies $\rect[\bsc,\bsz] \subseteq \cx$. The point $\bsc$ is called the anchor. 
\end{definition}

Figure~\ref{fig:extensible} shows some Sobol' extensible regions. 
Figure~\ref{fig:non_sobol} shows some sets which are not Sobol' extensible,
because no anchor point exists for them.
Sets like the right panel of Figure~\ref{fig:non_sobol} are of interest in QMC for 
functions with integrable singularities along the diagonal. 

\begin{propo}
If $\cx^{(j)}_j\subset\real^{d_j}$ is Sobol' extensible with anchor $\bsc_j$
for $j=1,\dots,s$, then $\prod_{j=1}^s\cx^{(j)}$ is Sobol' extensible
with anchor $\bsc = (\bsc_1,\dots,\bsc_s)$.
\end{propo}
\begin{proof}
Suppose that $\bsx \in\prod_{j=1}^s\cx^{(j)}$. 
We write $\bsx$ as $(\bsx_1,\dots,\bsx_s)$
where each $\bsx_j\in\cx^{(j)}$.
Then $\rect(\bsc,\bsx) \subset \prod_{j=1}^s\rect(\bsc_j,\bsx_j)\in\prod_{j=1}^s\cx^{(j)}$.
\end{proof}

Given points $\bsx,\bsy\in\real^m$ and a set $u\subseteq1{:}m$,
the hybrid point $\bsx_u{:}\bsy_{-u}$ is the point $\bsz\in\real^m$
with $z_j=x_j$ for $j\in u$ and $z_j=y_j$ for $j\not\in u$.
We will also require hybrid points $\bsx_u{:}\bsy_v{:}\bsz_w$
whose $j$'th component is that of $\bsx$ or $\bsy$ or $\bsz$
for $j$ in $u$ or $v$ or $w$ respectively, where those index
sets partition $1{:}m$.

A smooth function $f$ can be written as 
\begin{equation}\label{eq:mvfund}
f(\bsx) = \sum_{u \subseteq 1{:m}} \int_{[\bsc_u, \bsx_u]}\partial^uf(\bsc_{-u}{:}\bsy_u) \rd \bsy_u 
\end{equation}
where $\int_{[\bsc_u,\bsx_u]}$ denotes $\pm \int_{\rect[\bsc_u,\bsx_u]}$. The sign is negative if and only if $c_j > x_j$ holds for an odd number of indices $j \in u$.  The term for $u = \emptyset$ equals $f(\bsc)$ under a natural convention. 
Equation~\eqref{eq:mvfund} is a multivariable version of the fundamental
theorem of calculus.  For $m=1$ it simplifies to
$f(x) = f(c) + \int_c^x f'(y)\rd y$.

\begin{figure}[t]
\begin{center}
\includegraphics[scale = 0.7]{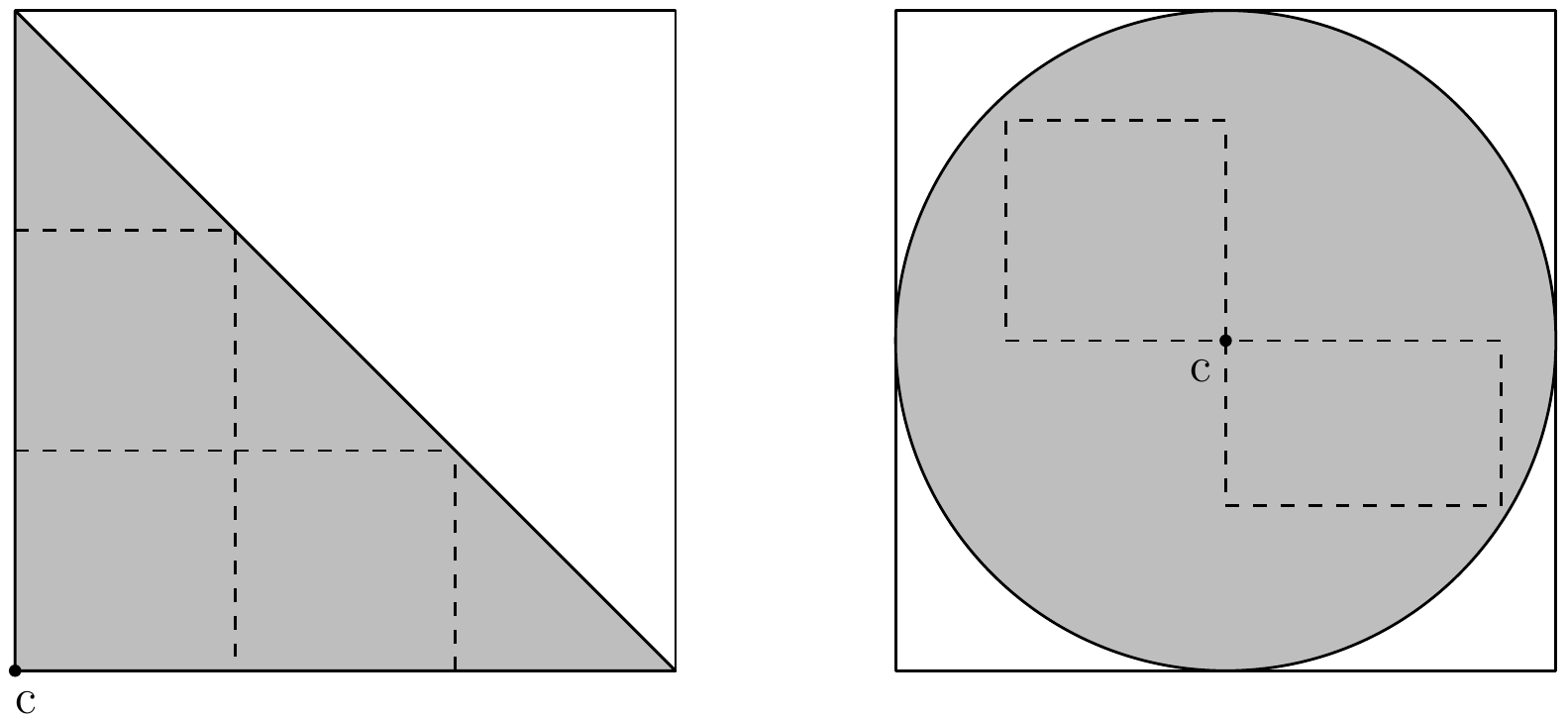}
\caption{\label{fig:extensible}
Sobol' extensible regions. At left, $\cx$ is the triangle with vertices
$(0,0)$, $(0,\sqrt{2})$, $(\sqrt{2},0)$ and the anchor  is $\bsc = (0,0)$. At right, $\cx$ is a circular 
disk centered its anchor $\bsc$. The dashed lines depict some rectangular hulls joining selected points to the anchor.}
\end{center}
\end{figure}

\begin{figure}[tb]
\begin{center}
\includegraphics[scale = 0.7]{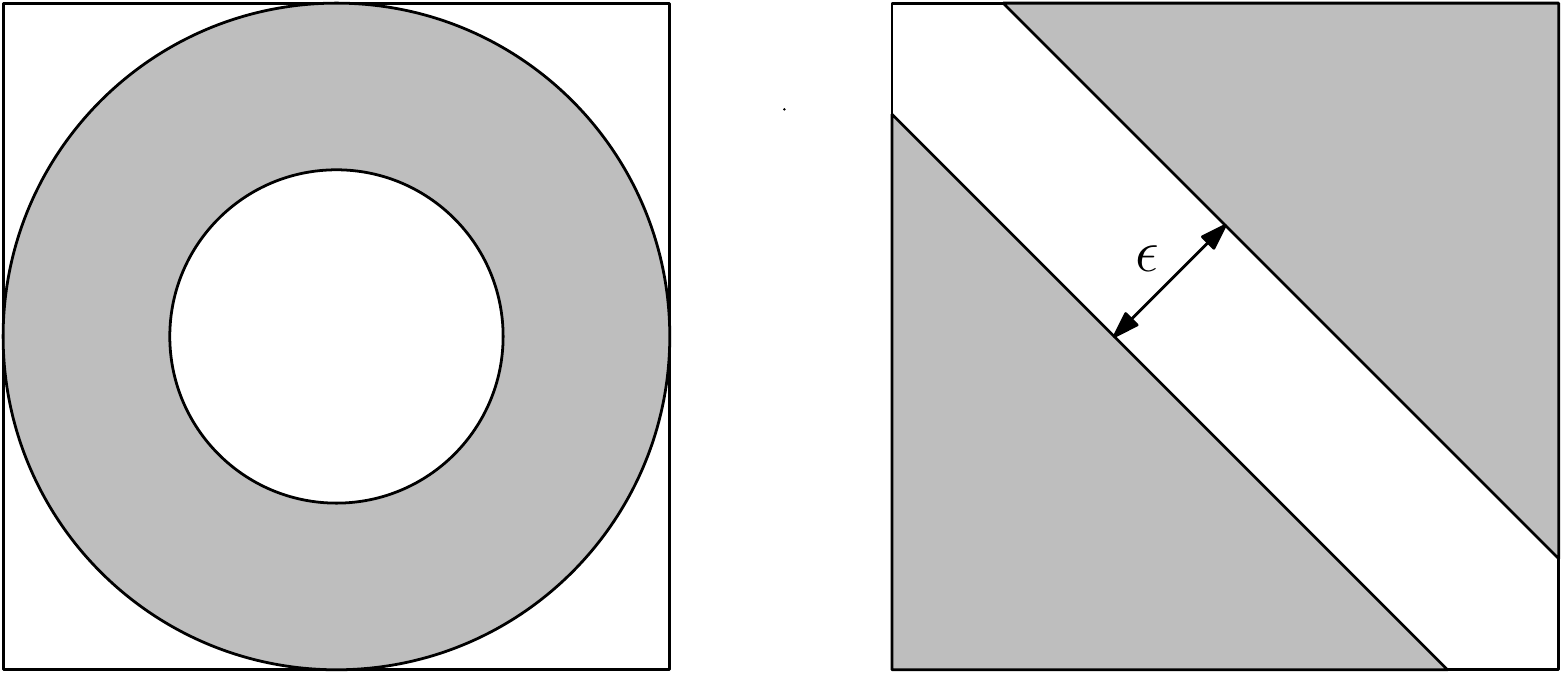}
\caption{\label{fig:non_sobol}
Non-Sobol' extensible regions. At left, $\cx$ is 
an annular region centered at the origin. 
At right, $\cx$ is the unit square exclusive of 
an $\epsilon$-wide strip centered on the diagonal.}
\end{center}
\end{figure}

\begin{definition}
Let $f$ be a smooth real-valued function on the Sobol' extensible
region $\cx\subset\real^m$.
The Sobol' extension of $f$ is the function $\tilde f:\real^m\to\real$
given by
\begin{equation}\label{sobol_ext}
\tilde{f}(\bsx) = \sum_{u\subseteq 1{:}m} \int_{[\bsc_u, \bsx_u]} \partial^u f(\bsc_{-u}{:}\bsy_u)
 1_{\bsc_{-u}{:}\bsy_u \in \cx} \rd \bsy_u 
\end{equation}
where $\bsc$ is the anchor of $\cx$.
\end{definition}

The Sobol' extension can be restricted to any domain $\cx'$
with $\cx\subset\cx'\subset\real^m$.
We usually use the Sobol' extension for $\tilde f$ from $\cx$ to $\rect(\cx)$.
This extension was used in \cite{sobo:1973} but not explained there.
An account of it appears in \cite{variation, owen2006quasi}. 
For $\bsx \in \cx$ the factor $1_{\bsc_{-u}{:}\bsy_u \in \cx}$ is always 1, making $\tilde{f}(\bsx) = f(\bsx)$  so that the term ``extension" is appropriate. 

Two simple examples serve to illustrate the Sobol' extension.
If $\cx = [0,1]$ and $f(x)=x$ on $0\le x\le 1$, then
the Sobol' extension of $f$ to $[0,\infty)$ is
$\tilde f(x) = \min(x,1)$.
If $\cx=[0,1]^2$ and $f(\bsx)=x_1x_2$ on $\cx$, 
then the Sobol' extension of $f$ to $[0,\infty)^2$ is
$\tilde f(\bsx)=\min(x_1,1)\times\min(x_2,1)$.
The Sobol' extension $\tilde f$ has a continuous mixed
partial derivative $\partial^{1{:}m}\tilde f$ for $\bsx$ in the interior
of $\cx$ and also in the interior of $\real^m\setminus\cx$
where $\partial^{1{:}m}\tilde f=0$ \citep{variation}.
As our examples show, $\partial^u\tilde f$ 
for $|u|>0$ may fail to exist at points of the boundary $\partial\cx$.  
A Sobol' extensible $\cx\subseteq\real^m$
has a boundary of $m$-dimensional measure $0$, so when forming
integrals of $\partial^u \tilde f$ we may ignore those points or simply
take those partial derivatives to be $0$ there.

The Sobol' extension has a useful property that we need.
It satisfies the multivariable fundamental theorem of
calculus, even though some of its partial derivatives
may fail to be continuous or even to exist everywhere. We can even
move the anchor from $\bsc$ to an arbitrary point $\bsz$.

\begin{theorem}\label{thm:ftocsobol}
Let $\cx\subset\real^m$ be a Sobol' extensible region
and let $f$ have a continuous mixed partial $\partial^{1{:}m}f$
on $\cx$.  Let $\tilde f$ be the Sobol' extension of $f$
and let $\bsz\in\real^m$. Then
\begin{equation}\label{eq:ftocsobol}
\tilde  f(\bsx) = \sum_{u \subseteq 1{:m}} \int_{[\bsz_u, \bsx_u]}\partial^u\tilde f(\bsz_{-u}{:}\bsy_u) \rd \bsy_u .
\end{equation}
\end{theorem}
\begin{proof}
Define
\begin{align}\label{eq:defg}
g(\bsx) = 
\sum_{v\subseteq 1{:}m}
\int_{[\bsz_v,\bsx_v]}\partial^v\tilde f(\bsz_{-v}{:}\bst_v)\rd\bst_v
\end{align}
where $\partial^v\tilde f$ is taken to be $0$ on
those sets of measure zero where it might not exist when $|v|>0$.
We need to show that $g(\bsx)=\tilde f(\bsx)$.
Now
\begin{align}\label{eq:innerg}
\tilde f(\bsz_{-v}{:}\bst_v)
= 
\sum_{u\subseteq1{:}m}
\int_{[\bsc_u,\bsz_{u\cap-v}{:}\bst_{u\cap v}]}
\partial_{\bsy}^uf(\bsc_{-u}{:}\bsy_u)
\cx(\bsc_{-u}{:}\bsy_u)\rd\bsy_u,
\end{align}
where for typographical convenience we have replaced $1_{\cdot\in\cx}$
by $\cx(\cdot)$. The subscript in $\partial^v_{\bsy}$ makes it easier
to keep track of the variables with respect to which that derivative is taken.
Substituting~\eqref{eq:innerg} into~\eqref{eq:defg},
we find that $g(\bsx)$ equals
\begin{align*}
&\sum_{v\subseteq 1{:}m}
\int_{[\bsz_v,\bsx_v]}\partial^v_{\bst}
\Biggl[\,\sum_{u\subseteq1{:}m}
\int_{[\bsc_u,\bsz_{u\cap-v}{:}\bst_{u\cap v}]} 
\partial^u_{\bsy}f(\bsc_{-u}{:}\bsy_u) 
\cx(\bsc_{-u}{:}\bsy_u)\rd\bsy_u 
\Biggr]\rd\bst_v\\
=&\sum_{v\subseteq 1{:}m}
\int_{[\bsz_v,\bsx_v]}\partial^v_{\bst}
\Biggl[\,\sum_{u\supseteq v}
\int_{[\bsc_u,\bst_{u}]}
\partial^u_{\bsy}f(\bsc_{-u}{:}\bsy_u) 
\cx(\bsc_{-u}{:}\bsy_u)\rd\bsy_u 
\Biggr]\rd\bst_v \\
=&\sum_{v\subseteq 1{:}m}
\int_{[\bsz_v,\bsx_v]}
\sum_{u\supseteq v}
\int_{[\bsc_{u-v},\bst_{u-v}]}
\partial^u_{\bsy}f(\bsc_{-u}{:}\bsy_{u-v}{:}\bst_v) 
\cx(\bsc_{-u}{:}\bsy_{u-v}{:}\bst_v) 
\rd\bsy_{u-v} \rd\bst_v. 
\end{align*} 
Now we introduce $w=u-v$ and rewrite the sum, getting
\begin{align*}
&\sum_{w\subseteq 1{:}m}\sum_{v\subseteq -w}
\int_{[\bsz_v,\bsx_v]}
\int_{[\bsc_{w},\bst_{w}]}
\partial^{w+v}_{\bsy}f(\bsc_{-w-v}{:}\bsy_{w}{:}\bst_v) 
\cx(\bsc_{-w-v}{:}\bsy_{w}{:}\bst_v) 
\rd\bsy_{w} \rd\bst_v\\
=&\sum_{w\subseteq 1{:}m}\sum_{v\subseteq -w}
\int_{[\bsz_v,\bsx_v]}
\partial^{v}_{\bsy}
\int_{[\bsc_{w},\bst_{w}]}
\partial^{w}_{\bsy}
f(\bsc_{-w-v}{:}\bsy_{w}{:}\bst_v) 
\cx(\bsc_{-w-v}{:}\bsy_{w}{:}\bst_v) 
\rd\bsy_{w} \rd\bst_v. 
\end{align*}
Any term above with $v\ne\emptyset$ vanishes. Therefore
\begin{align*}
g(\bsx)
&=
\sum_{w\subseteq 1{:}m}
\int_{[\bsz_\emptyset,\bsx_\emptyset]}
\int_{[\bsc_{w},\bst_{w}]}
\partial^{w}_{\bsy}
f(\bsc_{-w}{:}\bsy_{w}) 
\cx(\bsc_{-w}{:}\bsy_{w}) 
\rd\bsy_{w} \rd\bst_\emptyset\\
&=
\sum_{w\subseteq 1{:}m}
\int_{[\bsc_{w},\bst_{w}]}
\partial^{w}_{\bsy}
f(\bsc_{-w}{:}\bsy_{w}) 
\cx(\bsc_{-w}{:}\bsy_{w}) 
\rd\bsy_{w} \\
&=\tilde f(\bsx).\qedhere
\end{align*}
\end{proof}

\subsection{Whitney extension}

Here we assume that $\cx$ is a bounded closed set with non-empty interior,
not necessarily Sobol' extensible.
Sobol' extensible spaces may fail to have a non-empty interior, but outside
such odd cases, they are a subset of this class.
Non-Sobol' extensible regions like those in Figure~\ref{fig:non_sobol} are included. 
To handle domains $\cx$ of greater generality, we  require greater smoothness of $f$.

Let $\bsk\in\natu^m$ be any multi-index with $|\bsk| = k_1 + \ldots + k_{m}  \leq m$. 
We denote the $\bsk$-th order partial derivative as
\[
D_{\bsk} f(\bsx) = \frac{\partial^{|\bsk|}}{\partial x_1^{k_1}\cdots
\partial x_{m}^{k_{m}}} f(x_1, \ldots, x_{m}). 
\]

\begin{definition}
A real-valued function $f$ on $\cx\subset\real^m$ is in $C^m(\cx)$
if all partial derivatives of $f$ up to total order $m$ are continuous on $\cx$.
\end{definition}

Whitney's extension of a function in $C^m(\cx)$
to a function in $C^m(\rect(\cx))$ 
is given by the following lemma.
\begin{lemma}
Let $f\in C^m(\cx)$ for a bounded closed set $\cx\subset\real^m$ with
non-empty interior.
Then there exists a function $\tilde{f}\in C^m(\rect(\cx))$ with the following properties:
\begin{enumerate}
\item $\tilde{f}(\bsx) = f(\bsx)$ for all $\bsx \in \cx$,
\item $D_{\bsk}\tilde{f}(\bsx) = D_{\bsk}f(\bsx)$ for all $|\bsk| \leq m$ and 
$\bsx\in\cx$, and
\item $\tilde{f}$ is analytic on $\rect(\cx) \setminus \cx$.
\end{enumerate}  
\end{lemma}

\begin{proof}
The extension we need is the one provided by
\cite{whitney}. A function in $C^m(\cx)$ in the ordinary
sense is {\sl a fortiori} in $C^m(\cx)$ according to
Whitney's definition. We use the restriction of Whitney's
function to the domain $\rect(\cx)$.
\end{proof}

We will need one more condition on $\cx$. We require
the boundary of $\cx$ to have $m$-dimensional measure zero.
Then Theorem~\ref{thm:ftocsobol} in which the fundamental
theorem of calculus applies to $\tilde f$, holds also for the
Whitney extension.

\subsection{ANOVA components of extensions}

Here we show that the ANOVA components of our
smooth extensions are also smooth.  
We suppose that each $\cx^j\subset\real^{d_j}$
and we let $m=\sum_{j=1}^sd_j$.
The Cartesian product $\cx^\otos$ is now a subset of $\real^m$.

\begin{lemma}\label{smoothfu1}
Let $f$ be a smooth function on Sobol' extensible $\cx^\otos
\subset\real^m$ and for $u \subseteq 1{:}s$ 
let $f_u$ be the ANOVA component from \eqref{eq:anovau}. 
Then $f_u$ is smooth on $\cx^\otos$. 
\end{lemma}

\begin{proof}
We prove this by induction on $|u|$. Let $|u| = 0$, that is $u = \emptyset$. 
Then $f_u(\bsx) = \int_{\cx^\otos} f(\bsx) \rd \bsx$ which is a constant $\mu$ 
and is therefore smooth on $\cx^\otos$. 
Let us suppose that the hypothesis holds for $|u| = k - 1<s$ and we shall prove it for $|u| = k$. 

Fix any $u \subseteq\otos$ such that $|u| = k$. By \eqref{eq:anovau} we have,
 \[f_u(\bsx) = \int_{\cx^{-u}} f(\bsx) \rd \bsx_{-u} - \sum_{w \subset u} f_w(\bsx),\]
using the fact that $f_w(\bsx)$ does not depend on $x_j$ for $j\not\in w$. 
Each term in the summation is $f_w$ for $|w| \leq k-1$ and is therefore smooth by the induction hypothesis. So we only need to show that the first term is smooth. Fix any $v \subseteq 
1{:}m$. Now since $f$ is smooth, $\partial^v f(x)$ is continuous on $\cx^\otos$ and hence applying Leibniz's integral rule we have,
\[ \partial^v \int_{\cx_{-u}} f(\bsx) \rd \bsx_{-u} = \int_{\cx_{-u}} \partial^v f(\bsx) \rd \bsx_{-u}. 
\]
Now the right hand side is the integral of a continuous function  and is therefore a continuous function. Thus the induction hypothesis hold for $|u|=k$ completing the proof.
\end{proof}

\begin{lemma}\label{smoothfu}
Let $f\in C^{m}(\cx^\otos)$ for a bounded closed set $\cx^\otos\in\real^m$
and for  $u \subseteq 1{:}s$ let $f_u$ 
be the ANOVA component in~\eqref{eq:anovau}.
Then $f_u\in C^{m}(\cx^\otos)$. 
\end{lemma}

\begin{proof}
The proof goes along the same lines as Lemma \ref{smoothfu1}. 
We replace $v$ in that proof by any multi-index $\bsell$ with $|\bsell| \leq m$. 
Now since $f$ is smooth $D_{\bsell} f(\bsx)$ is continuous on $\cx^\otos$ and hence applying the Leibniz's integral rule we have,
\[ D_{\bsell}\int_{\cx^{-u}} f(\bsx) \rd \bsx_{-u} = \int_{\cx^{-u}} 
D_{\bsell} f(\bsx) \rd \bsx_{-u}. 
\]
Now the right hand side is the integral of a continuous function over certain variables and is therefore a continuous function. 
\end{proof}

Now for a smooth function $f$ defined on a product $\cx^\otos$
of Sobol' extensible sets, or on a product of more general
spaces but with the smoothness required for a Whitney extension,
there exists an extension $\tilde{f}$ on $\rect(\cx^\otos)$ such that 
    \begin{equation}\label{expansion}
    \tilde{f}(x) = \sum_{u \subseteq 1{:}m} \int_{[\bsc_u,\bsx_u]} 
\partial^u \tilde{f}(\bsc_{-u}{:}\bsy_u)\rd \bsy_u 
    \end{equation} for some point $\bsc \in \rect(\cx^\otos)$. 

\section{Scrambled net variance for smooth functions}\label{sec:snetvar}
Here we prove that the variance of averages over scrambled geometric nets is
$O(n^{-1 -2/d} \log(n)^{s-1})$, under smoothness and sphericity conditions. 
The proof is similar to the one in
\cite{owen2008local} for scrambled nets.
We begin with notation for some Cartesian products of cells.
For this section we assume that $d_j=d$ is a constant dimension for all
$j\in\otos$. 

Let $b$ be the common base for recursive splits $\bx_j$ of $\cx^{(j)}\subset\real^d$ for $j\in\otos$.
Let $\kappa = (k_1, \ldots, k_s)$ and $\tau = (t_1, \ldots, t_s)$ be $s$-vectors 
with $k_j\in\natu$ and $t_j\in\ints_{b^{k_j}}$.
Then we write
\[ \mathbb{B}_{u\kappa\tau} 
= \prod_{j \in u}\cx_{j,(k_j,t_j)}\prod_{j \not\in u} \cx^{(j)}
\]
and
$\wt{\mathbb{B}}_{u\kappa\tau} =\rect(\bb_{u\kappa\tau})$.
For $j=1,\dots,s$, let
$S_j= ((j-1)d+1){:}(jd)$ and then for $u\subseteq1{:}s$, define
\begin{align}
\label{su}
S_u &= \bigcup_{j \in u}S_j.
\end{align}
Now let 
$\mathbb{S}_u = \{T\subseteq S_u\mid T\cap S_j\ne\emptyset, 
\ \forall j\in u\}.$ These are the subsets of $S_u$ that
contain at least one element of $S_j$ for each $j\in u$.
There are $2^d-1$ non-empty subsets of $S_j$, and so
\begin{equation}\label{card}
|\bs_u| = (2^d - 1)^{|u|}.
\end{equation}


\begin{lemma}
\label{lem_bound1}
Suppose that $f$ is a smooth function on the Sobol' extensible
region $\cx^\otos\subseteq\real^{ds}$, with extension $\tilde  f$. 
Let each $\cx^{(j)}$ have a convergent recursive split in base $b$
whose $k$-level cells have diameter at most $Cb^{-k/d}$ for $1\le C<\infty$.
Let $u\subseteq 1{:}s$ and 
let $\kappa$ and $\tau$ be $|u|$-tuples with components $k_j\in\natu$ 
and $t_j\in\ints_{b^{k_j}}$, respectively for $j \in u$. 
Let $\psi_{u\kappa\tau\gamma}$ be the multiresolution basis function
\eqref{eq:multiresfunction}
defined by the splits of $\cx^\otos$.
Then
\begin{equation}\label{eq:lem_bound1}
|\langle f, \psi_{u\kappa \tau \gamma}\rangle|  \leq 
\Bigl( 2 - \frac{2}{b}\Bigr)^{|u|} b^{-\frac{|\kappa|}{2}(1+\frac{2}{d}) - \frac{|u|}{2}} \sum_{v \in \bs_u} 
d^{|v|/2}C^{|v|} \sup_{\bsy \in \wt{\bb}_{u\kappa\tau}} |\partial^v \tilde{f}_u(\bsy)|.
\end{equation}
If $f\in C^{ds}(\cx^\otos)$ with Whitney extension $\tilde f$,
where each $\cx^{(j)}$ is a bounded closed set with non-empty
interior and a boundary of measure zero,
then~\eqref{eq:lem_bound1} holds regardless of whether
$\cx^\otos$ is Sobol' extensible.
\end{lemma}

Note: Recall that we assume  that $\partial^v\tilde f_u$ takes the value $0$
in places where it is not well defined.  Alternatively one could
use the essential supremum instead of the supremum
in~\eqref{eq:lem_bound1}.
Later when we use $\Vert\cdot\Vert_\infty$ it will
denote the essential supremum of its argument.

\begin{proof}
The same proof applies to both smoothness assumptions.
From the definition we have
\[\begin{split}
\langle f, \psi_{u\kappa \tau \gamma}\rangle &= \langle f_u, \psi_{u\kappa \tau \gamma}\rangle \\
&= \int_{\cx^{-u}}\int_{\cx^u} f_u (\bsx) \psi_{u\kappa \tau \gamma} (\bsx) \rd \bsx_{u}\rd\bsx_{-u}\\
&= b^{-(|\kappa| + |u|)/2} \int_{\cx^u} f_u(\bsx) \prod_{j \in u} b^{k_j } \bigl(bN_{jk_jt_jc_j}(x_j) - W_{jk_jt_j}(x_j)\bigr) \rd \bsx_u.  
\end{split}
\]

By either Lemma \ref{smoothfu1} or Lemma \ref{smoothfu}, $f_u$ is smooth and we let $\tilde{f}_u$ be its extension. We know $\tilde{f}_u(\bsx) = f_u(\bsx)$ for all $\bsx \in \cx^\otos$. 
As the above integral is over $\cx^u$, we can write it as 
\begin{equation}\label{int_lemma_bound1}
b^{-(|\kappa| + |u|)/2} \int_{\cx^u} \tilde{f}_u(\bsx) \prod_{j \in u} b^{k_j} \left(bN_{k_jt_jc_j}(x_j) - W_{k_jt_j}(x_j)\right) \rd \bsx_u.
\end{equation}
Now $\tilde{f}_u$ is smooth on $\rect(\cx^u)$ and depends only on 
$\bsx_u$. Applying \eqref{expansion} we can write,
\begin{align}\label{eq:intv}
\tilde{f}_u(\bsx) = \sum_{v \subseteq S_u} \int_{[\bsz_v,\bsx_v]} 
\partial^v \tilde{f}_u(\bsz_{-v}{:}\bsy_v)\rd \bsy_v,
\end{align}
choosing to place the anchor $\bsz$ at the center of $\wt{\bb}_{u\kappa\tau}$.
Note that if $v \not\in \bs_u$, then there exists an index 
$j \in u$ such that $S_j\cap v=\emptyset$ 
and then the integral in~\eqref{eq:intv} above does not depend on $\bsx_j$ 
making it orthogonal to $bN_{jk_jt_jc_j}(\bsx_{j}) - W_{jk_jt_j}(\bsx_{j})$. 
Also the integrand in \eqref{int_lemma_bound1} is supported only for $\bsx_u \in  \bb_{u\kappa\tau}$. 
Putting these together we get,
\[ 
\begin{split}
&\,\quad b^{(|\kappa| + |u|)/2}|\langle f, \psi_{u\kappa \tau \gamma}\rangle|\\
&= \biggl| \int \sum_{v \in\bs_u} \int_{[\bsz_v,\bsx_v]} 
\partial^v \tilde{f}_u(\bsz_{-v}{:}\bsy_v)\rd \bsy_v \prod_{j \in u} b^{k_j} 
\bigl(bN_{jk_jt_jc_j}(\bsx_j) - W_{jk_jt_j}(\bsx_j)\bigr) \rd \bsx_u\biggr|\\
&\leq \sum_{v \in \bs_u} \sup_{\bsx_u \in  \mathbb{B}_{u\kappa\tau}} \biggl|  
\int_{[\bsz_v,\bsx_v]} \partial^v \tilde{f}_u(\bsz_{-v}{:}\bsy_v)\rd \bsy_v \biggr| \times \\
&\qquad\qquad \int_{\cx^u} \prod_{j \in u} b^{k_j} \left|bN_{jk_jt_jc_j}(\bsx_j) - W_{jk_jt_j}(\bsx_j)\right| \rd \bsx_u \\
&= \Bigl( 2 - \frac{2}{b}\Bigr)^{|u|}  \sum_{v \in \bs_u} \sup_{\bsx_u \in  \bb_{u\kappa\tau}} 
\biggl|  \int_{[\bsz_v,\bsx_v]} \partial^v \tilde{f}_u(\bsz_{-v}{:}\bsy_v)\rd \bsy_v \biggr|.
\end{split}
\]
Now since $\partial^v\tilde{f}_u$ is bounded we can write,
\[
\begin{split}
\biggl| \int_{[\bsz_v,\bsx_v]} \partial^v \tilde{f}_u(\bsz_{-v}{:}\bsy_v)\rd \bsy_v \biggr|
&\leq \vol(\rect[\bsz_v,\bsx_v])  \sup_{\bsy \in \tilde{\bb}_{u\kappa\tau}}  | \partial^v \tilde{f}_u(\bsy)|. 
\end{split}
\] 
Because $\bsz\in\tilde{\bb}_{u\kappa\tau}$ 
we have 
\begin{align*}
\vol(\rect[\bsz_v,\bsx_v]) 
&= \prod_{\ell\in v}|z_\ell-x_\ell|
\le C^{|v|}d^{|v|/2}\prod_{j\in v}
(b^{-k_j/d})^{|v\cap S_j|}\\
&\le (Cd^{1/2})^{|v|}b^{-|\kappa|/d}.
\end{align*}
The last inequality follows because $|v\cap S_j|\ge1$ for all $j\in u$
and also uses equation~\eqref{eq:diamrect} on the diameter of a bounding box.
Finally, putting it all together, we get
\[
|\langle f, \psi_{u\kappa \tau \gamma}\rangle| \leq \Bigl( 2 - \frac{2}{b}\Bigr)^{|u|} 
b^{-\frac{|\kappa|}{2}\left(1+\frac{2}{d}\right) - \frac{|u|}{2}} \sum_{v \in \bs_u} C^{|v|} d^{|v|/2}
\sup_{\bsy \in \wt{\bb}_{u\kappa\tau}} |\partial^v \tilde{f}_u(\bsy)|.\qedhere
\]
\end{proof}

The factor $d^{|v|/2}$ in the bound can be as large as $d^{s/2}$ in applications,
which may be quite large.  It arises as a $|v|$-fold product of ratios $\diam(\rect(\cdot))/\diam(\cdot)$
for cells.  For rectangular cells that product is $1$.  Similarly for cells that are `axis parallel' right-angle
triangles, the product is again $1$.

\begin{lemma} Let $h_u(\bsz) = \max_{v \in \bs_u} | \partial^v \tilde{f}_u(\bsz)|$ 
for $\bsz \in \wt{\bb}_{u\kappa\tau}$.
Under the conditions of Lemma \ref{lem_bound1},
\begin{equation*}
\sigma_{u\kappa}^2 \leq \biggl[\wt{C}^2\Bigl( 2 - \frac{2}{b}\Bigr)^3\biggr]^{|u|} 
b^{-2|\kappa|/d} \Vert h_u\Vert_\infty^2 
\end{equation*}
where $\wt{C} = d^{1/2}(2^d - 1)C^d$.
\end{lemma}

\begin{proof}
The supports of $\psi_{u\kappa\tau\gamma}$ and $\psi_{u\kappa\tau'\gamma'}$ are disjoint unless $\tau = \tau'$. Therefore
\[\nu_{u\kappa}^2(\bsx) = \sum_{\tau\mid u}\sum_{\gamma,\gamma'\mid u}\langle f, \psi_{u\kappa \tau \gamma}\rangle \langle f, \psi_{u\kappa \tau \gamma'}\rangle \psi_{u\kappa\tau\gamma}(\bsx) \psi_{u\kappa\tau\gamma'}(\bsx). 
\] 
Now
\[
\begin{split}
\sigma_{u\kappa}^2 &= \int_{\cx^\otos} \nu_{u\kappa}^2(\bsx) \rd \bsx \\
&= \sum_{\tau\mid u,\kappa}\sum_{\gamma,\gamma'\mid u}\langle f, \psi_{u\kappa \tau \gamma}\rangle \langle f, \psi_{u\kappa \tau \gamma'}\rangle \int_{\cx^\otos} \psi_{u\kappa\tau\gamma}(\bsx) \psi_{u\kappa\tau\gamma'}(\bsx) \rd \bsx\\
&= \sum_{\tau\mid u,\kappa}\sum_{\gamma,\gamma'\mid u}\langle f, \psi_{u\kappa \tau \gamma}\rangle \langle f, \psi_{u\kappa \tau \gamma'}\rangle \prod_{j\in u}( 1_{c_j = c_j'}  - b^{-1})\\
&\leq 
\Bigl( 2 - \frac{2}{b}\Bigr)^{2|u|} 
b^{-|\kappa|(1+\frac{2}{d}) - |u|} 
\sum_{\tau\mid u,\kappa} \Biggl(\sum_{v \in \bs_u} d^{|v|/2}C^{|v|}
\sup_{\bsy \in \wt{\bb}_{u\kappa\tau}} |\partial^v \tilde{f}_u(\bsy)|\Biggr)^2 
\sum_{\gamma,\gamma'\mid u} \prod_{j\in u}( 1_{c_j = c_j'}  - b^{-1}).
\end{split}\]

Some algebra shows that
$\sum_{\gamma,\gamma'\mid u} \prod_{j\in u}( 1_{c_j = c_j'}  - b^{-1}) = (2-2/b)^{|u|}$.
The supremum above is at most $\Vert h_u\Vert_\infty$.
From equation~\eqref{card}, we have $|\bs_u| = (2^d -1)^{|u|}$ and also $C^{|v|}
\le C^{d|u|}$ for $v\in \bs_u$.
There are $b^{|\kappa|}$ indices $\tau$ in the sum given $u$ and $\kappa$.
From these considerations,
\begin{align*}
\sigma^2_{u\kappa} & \le 
\Bigl( 2 - \frac{2}{b}\Bigr)^{3|u|} b^{-2|\kappa|/d}\Vert h_u\Vert_\infty^2(2^d-1)^{2|u|}d^{|u|}C^{2d|u|}\\
&\le
\biggl[\wt{C}^2\Bigl( 2 - \frac{2}{b}\Bigr)^3\biggr]^{|u|} 
b^{-2|\kappa|/d}\Vert h_u\Vert_\infty^2.\qedhere
\end{align*}
\end{proof}

\begin{theorem}\label{thm:main}
Let $\bsu_1, \ldots, \bsu_n$ be the points of a randomized $(t,m,s)$-net in base $b$. 
Let $\bsx_i = \phi(\bsu_i)\in\cx^\otos$ for $i = 1, \dots, n$ 
where $\phi$ is the componentwise application of the transformation
from convergent recursive splits in  base $b$.
Suppose as $n \rightarrow \infty$ with $t$ fixed, that all the gain coefficients of the net satisfy $\Gamma_{u\kappa} \leq G < \infty$. 
Then for a smooth $f$ on $\cx^\otos$,
\[\var(\hat{\mu}) =  O\left(\frac{(\log n)^{s-1}}{n^{1 + 2/d}}\right).\]
If $f\in C^{ds}(\cx^\otos)$ 
where each $\cx^{(j)}$ is a bounded closed set with non-empty
interior and a boundary of measure zero,
then~\eqref{eq:lem_bound1} holds regardless of whether
$\cx^\otos$ is Sobol' extensible.
\end{theorem}

\begin{proof}
We know from \eqref{var_multi} that
\[
\begin{split}
\var(\hat{\mu}) &\leq \frac{G}{n} \sum_{|u| > 0} \sum_{|\kappa| > (m - t - |u|)_+} \sigma_{u\kappa}^2\\
& \leq \frac{G}{n} \sum_{|u| > 0}  \left[\wt{C}^2\left( 2 - \frac{2}{b}\right)^3\right]^{|u|} \Vert h_u\Vert_\infty^2 \sum_{|\kappa| > (m - t - |u|)_+} b^{-2|\kappa|/d} \\
& \leq \frac{\wt{G}}{n} \sum_{|u| > 0} \sum_{|\kappa| > (m - t - |u|)_+} b^{-2|\kappa|/d}
\end{split}
\]
where 
\[ \wt{G} = G \left[\tilde{c}_d^2\left( 2 - \frac{2}{b}\right)^3\right]^{|u|} \max_{|u| > 0} \Vert h_u\Vert_\infty^2. 
\]
Since we are interested in the limit as $m \rightarrow \infty$, we may suppose that $m > s + t$. For such large $m$, we have
\[
\sum_{|\kappa| > (m - t - |u|)_+} b^{-2|\kappa|/d} = \sum_{r = m - t - |u| + 1}^\infty b^{-2r/d} \binom{r + |u| - 1}{|u| - 1}
\]
where the binomial coefficient is the number of $|u|$-vectors $\kappa$ of nonnegative integers that sum to $r$. Making the substitution $s = r - m + t + |u|$,
\[
\begin{split}
\sum_{|\kappa| > (m - t - |u|)_+} b^{-2|\kappa|/d} &= b^{(-m + t + |u|)2/d}\sum_{s = 1}^{\infty} b^{-2s/d} \binom{s + m - t - 1}{|u| - 1}\\
& \leq \frac{b^{(t+|u|)2/d}}{n^{2/d}(|u| - 1)!}\sum_{s = 1}^{\infty} b^{-2s/d} (s + m - t - 1)^{|u| - 1}\\
&= \frac{b^{(t+|u|)2/d}}{n^{2/d}(|u| - 1)!}\sum_{s = 1}^{\infty} b^{-2s/d} \sum_{j = 0}^{|u|-1} \binom{|u| - 1}{j}s^j(m - t - 1)^{|u| - 1 - j}\\
&= \frac{b^{(t+|u|)2/d}}{n^{2/d}} \sum_{j = 0}^{|u|-1} \frac{(m - t - 1)^{|u| - 1 - j}}{j!(|u| - 1 -j)!} \sum_{s = 1}^{\infty} b^{-2s/d} s^j \\
& \leq \frac{b^{(t+|u|)2/d}}{n^{2/d}} m^{|u| - 1} |u| \sum_{s = 1}^{\infty} b^{-2s/d} s^{|u| - 1}.
\end{split}
\]
Note by the ratio test it is easy to see that $\sum_{s = 1}^{\infty} b^{-2s/d} s^{|u| - 1}$ converges. Also as $m \leq \log_b(n)$ and $|u| \leq s$ we get
\[
\sum_{|\kappa| > (m - t - |u|)_+} b^{-2|\kappa|/d} = O\left(\frac{(\log n)^{s-1}}{n^{2/d}}\right).
\]
Plugging this back into the bound for the variance we get the desired result.
\end{proof}

\section{Discussion}\label{sec:disc}

Our integration of smooth functions
over an $s$-fold product of $d$-dimensional spaces
has root mean squared error  (RMSE) of $O(n^{-1/2-1/d}(\log(n))^{(s-1)/2})$.
Plain QMC might map $[0,1]^{sd}$ to $\real$.  If
the composition of the integrand with such a mapping
is in BVHK, then QMC attains an error rate
of $O(n^{-1}\log(n)^{sd-1})$.
Our mapping then has the advantage for $d=1$ and $2$.
When the composition is not in BVHK then QMC need not
even converge to the right integral estimate. Then scrambled
nets provide much needed assurance as well as error estimates.

When the composed integrand is smooth, then scrambled
nets applied directly to $[0,1]^{sd}$ would have an RMSE of
$O(n^{-3/2}\log(n)^{(sd-1)/2})$. That is a better asymptotic rate
than we attain here, and it might really be descriptive of finite sample sizes
even for very large $sd$, if the composite integrand were of
low effective dimension \citep{cafmowen}.
If however, the composed integrand is in $L^2$ but is not smooth, then scrambled
nets applied in $sd$ dimensions would have an RMSE of
$o(n^{-1/2})$ but not necessarily better than that. Our proposal is
then materially better for small $d$.

The composed integrand that we actually use is {\sl not}
smooth on $[0,1]^s$.  It generally has discontinuities at all $b$-adic
fractions $t/b^k$ for any of the components of $\bsu$.
For example in the four-fold split of Figure~\ref{fig:base234},
an $\epsilon$ change in $\bsu$ can move a point from
the top triangle to the right hand triangle. These are however
axis-aligned discontinuities.  \cite{wang2011quasi} call these
QMC-friendly discontinuities.  They don't induce infinite variation.

We have used nested uniform scrambles. The same results
apply to other scrambles, notably the linear scrambles of
\cite{mato:1998}. Those scrambles are less space-demanding
than nested uniform scrambles. A central limit theorem
applies to averages over nested uniform scrambles \citep{loh:2003},
but has not been shown for linear scrambles.
\cite{Hong2003335} find that nested uniform scrambles have
stochastically smaller values of a squared discrepancy measure.

The splits we used allowed overlaps on sets of measure zero.
We could also have relaxed 
$\cx = \cup_{a=0}^{b-1}\cx_a$
to $\vol(\cx \setminus\cup_{a=0}^{b-1}\cx_a)=0.$
That could cause the deterministic construction to fail
to be a weak geometric $(t,m,s)$-net but the scrambled
versions would still be geometric $(t,m,s)$-nets with
probability one.

Our main result was proved assuming that all $d_j=d$.
We can extend it to unequal $d_j$ by taking $d=\max_{j\in\otos}d_j$.
To make the extension, one can add $d_j-d$ `do nothing'
dimensions to $\cx^{(j)}$. The splits never take place along those
dimensions, so the cells become cylinder sets and the function
does not depend on the value of those components. 
We can make the extent of those do-nothing dimensions as small
as we like to retain control of the diameter of the splits and then
apply Theorem~\ref{thm:main}.

\bibliographystyle{apalike}
\bibliography{qmc,qmc2}
\end{document}